\documentclass[11pt,a4paper, reqno]{amsart}

\usepackage{a4wide}

\usepackage{amsfonts,amsthm,amssymb,amsmath,amscd,amsopn,mathabx}
\usepackage[pdftex]{graphicx}
\usepackage{graphics}
\usepackage[matrix,arrow]{xy}
\usepackage[active]{srcltx}

\usepackage{enumitem, hyperref}\hypersetup{colorlinks}

\makeatletter
\def\reflb#1#2{\begingroup
    #2%
    \def\@currentlabel{#2}%
    \phantomsection\label{#1}\endgroup
}
\makeatother
\usepackage{color}%\textcolor{red}{Test}

\definecolor{darkred}{rgb}{1,0,0} %can change the intensity in [0,1]
\definecolor{darkgreen}{rgb}{0,0.8,0}
\definecolor{darkblue}{rgb}{0,0,1}

\hypersetup{colorlinks,
linkcolor=darkblue,
filecolor=darkgreen,
urlcolor=darkred,
citecolor=darkblue}

\newtheorem{thm}{Theorem}
\numberwithin{thm}{section}
\numberwithin{equation}{section}
\newtheorem{theorem}[thm]{Theorem}
\newtheorem*{theorem*}{Theorem}

\newtheorem{corollary}[thm]{Corollary}
\newtheorem*{corollary*}{Corollary}
\newtheorem{lemma}[thm]{Lemma}

\newtheorem{proposition}[thm]{Proposition}

\newtheorem*{conjecture*}{Conjecture}

\newtheorem*{question*}{Question}
\newtheorem{definition}[thm]{Definition}
\newtheorem*{definition*}{Definition}

\newtheorem*{definitions*}{Definitions}

\newtheorem*{rem*}{Remark}
\theoremstyle{remark}
\newtheorem{remark}[thm]{Remark}

\newtheorem*{remark*}{Remark}

\newtheorem*{remarks*}{Remarks}
\newtheorem*{example*}{Example}

\newtheorem*{examples*}{Examples}

\newcommand{\R}{\mathbb{R}}
\newcommand{\Z}{\mathbb{Z}}
\newcommand{\Q}{\mathbb{Q}}
\newcommand{\C}{\mathbb{C}}
\newcommand{\N}{\mathbb{N}}

\def\CP{{\mathbb C}P}
\def\RP{{\mathbb R}P}

\newcommand{\id}{\mathit{id}}
\newcommand{\Id}{\mathit{Id}}

\newcommand{\tspan}{\text{span}}

\newcommand{\ep}{\epsilon}
\newcommand{\ga}{\gamma}
\newcommand{\Ga}{\Gamma}
\newcommand{\al}{\alpha}
\newcommand{\vr}{\varphi}
\newcommand{\om}{\omega}

\newcommand{\Cont}{Cont}
\newcommand{\Diff}{\textit{Diff}\,}

\newcommand{\Sp}{\mathrm{Sp}}

\newcommand{\sign}{\operatorname{sign}}

\def\cz{{\mu}}
\def\P{{\mathcal P}}

\def\S{{\mat\HCal S}}

\def\mi{{\widehat\mu}}

\def\Bott{{\mathcal B}}
\def\bga{{\bar\gamma}}
\def\balpha{{\bar\alpha}}
\def\la{{\langle}}
\def\ra{{\rangle}}

\def\S{{\mathcal S}^{\epsilon'}}
\def\bump{{\mathfrak b}}
\def\cz{{\mu}}

\def\bPsi{{\bar\Psi}}

\def\bvr{{\bar\varphi}}

\def\bPhi{{\bar\Phi}}
\def\S{{\Sigma}}
\def\tGa{{\widetilde\Ga}}
\def\tPhi{{\widetilde\Phi}}
\def\lg{\langle}
\def\rg{\rangle}
\def\cGa{\Gamma^\C}
\def\sf{\text{sf}}

\begin{document}

\title[Dynamical convexity and closed orbits on symmetric spheres]{Dynamical convexity and closed orbits\\ on symmetric spheres}

\author[Viktor Ginzburg]{Viktor L. Ginzburg}
\address{Department of Mathematics, UC Santa Cruz, Santa Cruz, CA
  95064, USA} \email{ginzburg@ucsc.edu} 

\author[Leonardo Macarini]{Leonardo Macarini}
\address{Center for Mathematical Analysis, Geometry and Dynamical Systems,
Instituto Superior T\'ecnico, Universidade de Lisboa, 
Av. Rovisco Pais, 1049-001 Lisboa, Portugal}
\email{macarini@math.tecnico.ulisboa.pt}

\subjclass[2010]{53D40, 37J10, 37J55} \keywords{Closed orbits,
Conley-Zehnder index, Reeb flows, dynamical convexity, equivariant symplectic homology}

%\date{\today} 

\thanks{Viktor Ginzburg was partially supported by Simons Collaboration Grant 581382. Leonardo Macarini was partially supported by FCT/Portugal through the projects UID/MAT/04459/2019 and PTDC/MAT-POR/29447/2017.}

\begin{abstract} 
The main theme of this paper is the dynamics of Reeb flows with symmetries on the standard contact sphere. We introduce the notion of strong dynamical convexity for contact forms invariant under a group action, supporting the standard contact structure, and prove that  in dimension $2n+1$ any such contact form satisfying a condition slightly weaker than strong dynamical convexity has at least $n+1$ simple closed Reeb orbits.  For contact forms with antipodal symmetry, we prove that  strong dynamical convexity is a consequence of ordinary convexity.  In dimension five or greater, we construct examples of antipodally symmetric dynamically convex contact forms which are not strongly dynamically convex, and thus not contactomorphic to convex ones via a contactomorphism commuting with the antipodal map. Finally, we relax this condition on the contactomorphism furnishing a condition that has non-empty $C^1$-interior.
\end{abstract}

\maketitle

\tableofcontents

\section{Introduction and main results}

\subsection{Introduction}

In this paper we focus on the dynamics of Reeb flows on hypersurfaces  with symmetry in the standard symplectic vector space. In this setting, we study the multiplicity problem for closed Reeb orbits (without any non-degeneracy requirements) and the relation between convexity and dynamical convexity. For hypersurfaces with symmetry, dynamical convexity can be replaced by a more subtle and stronger condition taking into account different roles of symmetric and asymmetric orbits, which we refer to as ``strong dynamical convexity". We show that under this condition the number of simple closed Reeb orbits on a hypersurface in $\R^{2n+2}$ is no less than $n+1$. While in general the relation between convexity and strong dynamical convexity is rather involved, convexity implies strong dynamical convexity when the hypersurface is symmetric with respect to the antipodal map. Furthermore, we construct an example of an antipodally symmetric  hypersurface which is dynamically convex but not strongly dynamically convex, and thus not contactomorphic to a symmetric convex hypersurface via a contactomorphism commuting with the antipodal map. Finally, we relax this hypothesis on the contactomorphism furnishing a condition that has non-empty $C^1$-interior.

To elaborate, one of the most fundamental problems in Hamiltonian dynamics is the multiplicity question for simple (i.e., non-iterated) closed orbits of Reeb flows on the standard contact sphere $(S^{2n+1},\xi_{std})$. A long standing conjecture is that there are at least $n+1$ simple closed orbits for \emph{any} contact form $\alpha$ on $(S^{2n+1},\xi_{std})$. This was proved for $n=1$ by Cristofaro-Gardiner and Hutchings \cite{CGH} (in the more general setting of Reeb flows in dimension three) and independently by Ginzburg, Hein, Hryniewicz and Macarini \cite{GHHM}; see also \cite{LL} where an alternate proof was given using a result from \cite{GHHM}. In higher dimensions, the question is completely open without additional assumptions on the hypersurface, such as convexity or certain index requirements or non-degeneracy of closed Reeb orbits.

There is a natural bijection between contact forms $\alpha$ on $(S^{2n+1},\xi_{std})$ and starshapped hypersurfaces $\S_\al$ in $\R^{2n+2}$ so that $\al$ becomes the restriction of the Liouville form to $\S_\al$. We say that a contact form $\al$ on $S^{2n+1}$ is \emph{convex} if $\S_\al$ bounds a strictly convex subset.  Let us denote by $\P$ the set of simple closed Reeb orbits for $\alpha$. When $\alpha$ is convex, a remarkable result due to Long and Zhu \cite{LZ} asserts that $\#\P\geq \lfloor n/2 \rfloor +1$. This result was improved when $n$ is odd by Wang \cite{Wa}, furnishing the lower bound $\#\P\geq \lceil n/2 \rceil +1$.

The convexity requirement is not natural from the point of view of contact topology since it is not a condition invariant under contactomorphisms. An alternative notion, introduced by Hofer, Wysocki and Zehnder \cite{HWZ}, is \emph{dynamical convexity}.  A contact form $\alpha$ on $S^{2n+1}$ is called dynamically convex if every closed Reeb orbit $\ga$ of $\alpha$ has Conley-Zehnder index $\cz(\ga)$ greater than or equal to $n+2$. Clearly dynamical convexity is invariant under contactomorphisms and it is not hard to see that convexity implies dynamical convexity. When $\alpha$ is dynamically convex, the first author and G\"urel \cite{GG} and, independently, Duan and Liu \cite{DL}, proved that $\#\P\geq \lceil n/2 \rceil +1$, showing that the lower bound established by Long, Zhu and Wang in \cite{LZ,Wa}  in the convex case holds for dynamically convex hypersurfaces. 

Another important question is that of the relation between dynamical convexity and convexity: Is every dynamically convex hypersurface symplectomorphic or contactomorphic to a convex one? There seems to be no reason to expect the answer to be affirmative. For convexity can be easily seen to have dynamical consequences going far beyond periodic orbits (e.g., index positivity) while dynamical convexity is a notion strictly limited to the index behavior of closed Reeb orbits. However, this question in general is very hard and no results in this direction have been obtained so far.

As we will show in this work, the situation for both questions changes dramatically when the hypersurface is symmetric, e.g., with respect to the antipodal map. For instance, when $\alpha$ satisfies this condition and is convex, Liu, Long and Zhu showed in \cite{LLZ} that $\#\P\geq n+1$, proving the aforementioned conjecture for this class of contact forms. As has been mentioned above, for hypersurfaces with symmetry a natural replacement of dynamical convexity is strong dynamical convexity introduced in this paper and taking into account the action of the symmetry group; see Definition \ref{def:SDC}. We prove in Theorem \ref{thm:multiplicity}, that if $\alpha$ is strongly dynamically convex then $\#\P\geq n+1$. Then, in Theorem \ref{thm:c=>sdc1}, we show that for hypersurfaces with antipodal symmetry, strong dynamical convexity is a consequence of convexity, obtaining, in this way, a generalization of the result from \cite{LLZ}. Using Theorem \ref{thm:c=>sdc1}, we also show that the question of the relation between dynamical convexity and convexity can be completely answered for antipodally symmetric hypersurfaces in dimension bigger than three. Indeed, we construct an example of an antipodally symmetric  dynamically convex hypersurface of dimension five or greater which is not strongly dynamically convex, and hence not equivalent to a symmetric convex hypersurface via a contactomorphism preserving the symmetry. This is our main result and the content of Theorem \ref{thm:example}. When $n$ is odd, the condition that the contactomorphism preserves the symmetry can be relaxed to a condition that has non-empty $C^1$-interior; see Theorem \ref{thm:example open}. This requires a non-trivial generalization of Theorem \ref{thm:c=>sdc1}, given by Theorem \ref{thm:c=>sdc2}.

We conclude this section by briefly touching upon the role of degenerate orbits in these results and constructions. When all simple closed Reeb orbits are non-degenerate the strong dynamical convexity condition reduces to ordinary dynamical convexity. Hence, degeneracy is crucial to our construction of a dynamically convex, but not convex, symmetric hypersurface. With multiplicity results the situation is more subtle. The existence of $n+1$ simple closed Reeb orbits on a dynamically convex hypersurface with non-degenerate Reeb flow (i.e. such that every closed Reeb orbit, including iterated ones, is non-degenerate) has been established in \cite{GK} as a consequence of the so-called common index jump theorem proved in \cite{LZ}; see also \cite{AM,GG} and Theorem \ref{thm:IRT}. (In fact, this lower bound is now known to hold under much less restrictive index conditions than dynamical convexity; see \cite{DLLW, GGM}.) However, even without symmetry and when all simple closed Reeb orbits are non-degenerate (and thus strong dynamical convexity is equivalent to dynamical convexity),  iterated orbits may degenerate. In this case, our multiplicity result (Theorem \ref{thm:multiplicity}) gives some new information; see Corollaries \ref{cor:multiplicity-notsym} and \ref{cor:semisimple}.

\subsection{Strong dynamical convexity}
\label{sec:sdc}

Before defining strong dynamical convexity, let us set the sign conventions used throughout this paper and introduce some terminology and notation.

\vskip .2cm
\noindent {\bf Sign and index conventions.} 
Given a symplectic manifold $(M,\om)$ and a Hamiltonian $H_t: M \to \R$, we take Hamilton's equation to be $i_{X_{H_t}}\om=-dH_t$. A compatible almost complex structure $J$ is defined by the condition that $\om(\cdot,J\cdot)$ is a Riemannian metric. Throughout this work, the Conley-Zehnder index $\mu$ is normalized so that when $Q$ is a small positive definite quadratic form the path $\Ga: [0,1] \to \Sp(2n)$ generated by $Q$ and given by $\Ga(t)=\exp(tJQ)$ has $\cz(\Ga)=n$. This convention is consistent with the one used in Sections 4 and 5 of \cite{GG}. We also take the canonical symplectic form on $\R^{2n}$ to be $\sum dq_i \wedge dp_i$. For degenerate paths, the Conley--Zehnder index $\mu$ is defined as the lower semi-continuous extension of the Conley--Zehnder index from the paths with non-degenerate endpoint. More precisely,
$$
\mu(\Gamma)=\liminf_{\tilde{\Gamma}\to\Gamma}\mu(\tilde{\Gamma}),
$$
where $\tilde{\Gamma}$ is a small perturbation of $\Gamma$ with non-degenerate endpoint.
\vskip .2cm

Next, let $A \in \Sp(2d)$ be a totally degenerate symplectic matrix, i.e., all eigenvalues of $A$ are equal to one. One can show that then $A=\exp(JQ)$, where $Q$ is a symmetric matrix with all eigenvalues zero. Examining Williamson's normal forms  (see, e.g.,  \cite[Appendix 6]{Ar}) it is easy to see that the quadratic form $Q$ can be symplectically decomposed into a sum of terms of four types:
\begin{itemize}
\item the identically zero quadratic form on $\R^{2d}$,
\item the quadratic form $Q_0 = p_1q_2 +p_2q_3+\cdots+p_{d-1}q_d$ in Darboux coordinates on $\R^{2d}$, where $d \geq 3$ is odd,
\item the quadratic forms $Q_\pm = \pm(Q_0 + p_{d}^2/2)$ on $\R^{2d}$ for any $d$, where $Q_0=0$ when $d=1$.
\end{itemize}
We define $b_\pm(A)$ as the number of $Q_\pm$ terms in this decomposition. These are linear symplectic invariants of $A$ although this is not obvious; cf.\ \cite{GG}. Furthermore, given a symplectic matrix $P$, let $V$ be the subspace whose complexification is the generalized eigenspace of the eigenvalue one. Define $b_\pm(P)=b_\pm(P|_V)$. For a closed Reeb orbit $\ga$, we set $b_\pm(\ga)=b_\pm(P_\ga)$, where $P_\ga$ is the linearized Poincar\'e return map of $\ga$.

Throughout the paper, we will always consider the sphere endowed with the standard contact structure. Let $G$ be a group acting on $S^{2n+1}$ and let $\alpha$ be a $G$-invariant contact form. A closed Reeb orbit $\ga$ of $\alpha$ is called symmetric if $g(\ga(\R))=\ga(\R)$ for every $g \in G$. Given a $G$-invariant contact form, write $\P = \P_s \cup \P_{ns}$ where $\P_s$ (resp.\ $\P_{ns}$) is the set of symmetric (resp.\ non-symmetric) simple closed orbits.

\begin{definition}
\label{def:SDC}
A contact form $\alpha$ is \emph{strongly dynamically convex} if
\[
\cz(\ga) \geq n+2
\]
for every $\ga \in \P$, and
\[
\cz(\ga_s) + b_-(\ga_s) - b_+(\ga_s) \geq n+2
\]
for every $\ga_s \in \P_s$. 
\end{definition}

Note that this notion is natural from the point of view of contact topology since it is invariant under contactomorphisms in the following sense. Let  $\vr: S^{2n+1} \to S^{2n+1}$ be a contactomorphism and consider a $G$-action on $S^{2n+1}$. If $\alpha$ is a $G$-invariant strongly dynamically convex contact form then $\vr^*\alpha$ is strongly dynamically convex under the conjugated $G$-action $\vr^{-1}g\vr$.

\subsection{Main results}

Our first result establishes that every $G$-invariant contact form on $S^{2n+1}$ satisfying a condition slightly weaker than strong dynamical convexity carries at least $n+1$ simple periodic orbits. (Note the difference between the lower bound $n+1$ in the theorem below and the lower bound $n+2$ in Definition \ref{def:SDC}.) More precisely, let $\P^+ \subset \P$ be the set of simple periodic orbits with positive mean index. Write $\P^+ = \P^+_s \cup \P^+_{ns}$ where $\P^+_s$ (resp. $\P^+_{ns}$) is the set of symmetric (resp.\ non-symmetric) closed orbits with positive mean index.

\begin{theorem}
\label{thm:multiplicity}
Let $G$ be a group acting on $S^{2n+1}$ and $\alpha$ a $G$-invariant contact form. Suppose that $\alpha$ is strongly dynamically convex or more generally
\[
\cz(\ga) \geq n+1
\]
for every $\ga \in \P^+$ and
\[
\cz(\ga_s) + b_-(\ga_s) - b_+(\ga_s) \geq n+1
\]
for every $\ga_s \in \P^+_s$. Then $\#\P^+ \geq n+1$. 
\end{theorem}

\begin{remark}
In fact, the result we prove is more general than stated. We do not really need the action of a group: it is enough to consider a \emph{subset} of $\Diff(S^{2n+1})$ that is not necessarily a subgroup. More precisely, given a subset $S \subset \Diff(S^{2n+1})$ we say that $\alpha$ is $S$-invariant if $\vr^*\alpha=\alpha$ for every $\vr \in S$. A closed orbit $\ga$ of $\alpha$ is symmetric if $\vr(\ga(\R))=\ga(\R)$ for every $\vr \in S$. Then the previous theorem holds for $S$-invariant contact forms; see Section \ref{sec:pf-multiplicity}. Moreover,  the definition of strong dynamical convexity clearly extends to this more general context. 
\end{remark}

In particular, when the action is trivial, we have the following consequence of the theorem.

\begin{corollary}
\label{cor:multiplicity-notsym}
Let $\alpha$ be a contact form on $S^{2n+1}$. Suppose that
\[
\cz(\ga) \geq n+1\quad \text{and}\quad \cz(\ga) + b_-(\ga) - b_+(\ga) \geq n+1
\]
for every $\ga \in \P^+$. Then $\#\P^+ \geq n+1$.
\end{corollary}

Clearly, $b_\pm(P)=0$ for every symplectic matrix $P$ such that the eigenvalue one is semisimple. Therefore, we obtain the following corollary.

\begin{corollary}
\label{cor:semisimple}
Let $\alpha$ be a dynamically convex contact form on $S^{2n+1}$. Suppose that the eigenvalue one of the linearized Poincar\'e map of every degenerate simple closed orbit of $\alpha$ is semisimple. Then $\#\P \geq n+1$. In particular, this lower bound holds if every simple closed orbit of $\alpha$ is non-degenerate.
\end{corollary}

Unfortunately, when the action is trivial, it is probably not true that every convex contact form is strongly dynamically convex. However, for some natural non-trivial actions, convex invariant contact forms are strongly dynamically convex. More precisely, we have the following result.

\begin{theorem}
\label{thm:c=>sdc1}
Let $\alpha$ be a convex contact form on $S^{2n+1}$ invariant under the antipodal map. Then $\alpha$ is strongly dynamically convex.
\end{theorem}

This theorem is proved in Section \ref{sec:c=>sdc}. As a consequence, we recover the aforementioned result from \cite{LLZ}:

\begin{corollary}
Every convex contact form on $S^{2n+1}$ invariant under the antipodal map has at least $n+1$ simple closed orbits.
\end{corollary}

Note that there is also a similar multiplicity question in the setting where the involution is anti-symplectic rather than symplectic, which is closely related to the  Seifert conjecture on brake orbits; see, e.g., \cite{KKK} for further references and a Floer theoretic approach to the problem.

As noticed before, a convex contact form on $S^{2n+1}$ is dynamically convex. An important question in contact topology is whether every dynamically convex contact form is contactomorphic to a convex one.  Our next result, proved in Sections \ref{sec:proof main result} and \ref{sec:technical},  shows that this is not true if the underlying contactomorphism is required to commute with the antipodal map.

\begin{theorem}
\label{thm:example}
For $n\geq 2$, there exists an antipodally symmetric contact form $\alpha$ on $S^{2n+1}$ which is dynamically convex but not strongly dynamically convex and thus not contactomorphic to a convex contact form via a contactomorphism commuting with the antipodal map.
\end{theorem}

\begin{remark}
\label{rmk:C1 close}
The contact form $\alpha$ can be chosen arbitrarily $C^1$-close to the Liouville form restricted to the round sphere in $\R^{2n+2}$; see Remark \ref{rmk:H C1 close to 1}.
\end{remark}

The fact that the form $\alpha$ in the theorem is not equivalent to a convex form via an equivariant contactormorphism is a consequence of the invariance of strong dynamical convexity under contactomorphisms. Indeed, from the invariance, we see that if $\vr: S^{2n+1} \to S^{2n+1}$ is a contactomorphism then $\vr^*\alpha$ is strongly dynamically convex under the conjugated $\Z_2$-action. If $\vr$ commutes with the antipodal map then the conjugated action is also generated by the antipodal map and consequently $\vr^*\alpha$ cannot be convex by Theorem \ref{thm:c=>sdc1}.

At this stage we do not know if the requirement that the contactomorphism commutes with the antipodal map is essential. However, when $n$ is odd, we can improve the previous theorem furnishing a condition specifying a set of contactomorphisms with \emph{non-empty $C^1$-interior}. More precisely, we have the following result, proved in Section \ref{sec:example open}.

\begin{theorem}
\label{thm:example open}
If $n\geq 3$ is an odd integer, the contact form $\alpha$ furnished by Theorem \ref{thm:example} satisfies the following. Let $S \subset \Cont(S^{2n+1})$ be the subset of contactomorphisms that commute with the antipodal map. Then there exists a $C^1$-open subset $U \subset \Cont(S^{2n+1})$, whose closure (in the $C^1$-topology) contains $S$, such that $\vr^*\alpha$ cannot be convex for any $\vr$ in $\widebar U$.
\end{theorem}

\subsection{Organization of the paper and acknowledgments}\hspace*{\fill} \\

\noindent \textbf{Organization of the paper.}
The rest of the paper is organized as follows. The background on the Conley--Zehnder index theory necessary for this work is presented in Section \ref{sec:background}. In the same section, we prove a comparison result (Theorem \ref{thm:comparison}) that plays a major role in the proof of Theorem \ref{thm:example open}. In Section \ref{sec:pf-multiplicity}, we prove our multiplicity result, Theorem \ref{thm:multiplicity}. Theorem \ref{thm:c=>sdc1}, which asserts that antipodally symmetric convex contact forms are strongly dynamically convex, is proved in Section \ref{sec:c=>sdc}. Our main result, Theorem \ref{thm:example}, addressing the problem of the relation between convexity and dynamical convexity, is established in Sections \ref{sec:proof main result} and \ref{sec:technical}. Finally, Section \ref{sec:example open} is devoted to the proof of Theorem \ref{thm:example open}.

\vskip .2cm
\noindent \textbf{Acknowledgments.}
The second author is grateful to Miguel Abreu and Yiming Long for useful discussions.

\section{Basic background on index theory for symplectic paths}
\label{sec:background}

\subsection{Index recurrence theorem}
A crucial ingredient for distinguishing simple and iterated orbits in the proof of Theorem \ref{thm:multiplicity} is the following  combinatorial result, taken from \cite[Theorem 5.2]{GG}, addressing the index behavior under iterations. This result is essentially contained, although in a different form, in \cite{DLW,LZ} as the so-called common index jump theorem.

\begin{theorem}[\cite{GG}]
\label{thm:IRT}
Let $\Phi_i: [0,1] \to \Sp(2n)$, with $i \in \{1,\dots,r\}$, be a finite collection of symplectic paths with positive mean index, starting at the identity. Then for any $\eta>0$ and any $\ell_0\in\N$, there exist positive integers $d,k_1,\dots,k_r$ such that, for all $i$ and any integer $\ell \in \Z$ in the range $1\leq \ell\leq \ell_0$, we have
\begin{itemize}
\item[\rm{(i)}] $\big|\mi(\Phi^{k_{i}}_i)-d\big|<\eta$, where $\mi$ is the mean index,
\item[\rm{(ii)}] $\cz(\Phi^{k_{i}+\ell}_i)= d + \cz(\Phi^\ell_i)$,
\item[\rm{(iii)}] $\cz(\Phi^{k_{i}-\ell}_i)= d + \cz(\Phi^{-\ell}_i) + \big(b_+(\Phi^{\ell}_i(1))-b_-(\Phi^{\ell}_i(1))\big)$.
\end{itemize}
Furthermore, for any $N\in \N$ we can make all $d,\,k_1,\,\dots,\,k_r$ divisible by $N$.
\end{theorem}

\subsection{Bott's function}
\label{sec:Bott}
Let $\Ga: [0,T] \to \Sp(2n)$ be a symplectic path starting at the identity and $P:=\Ga(T)$ its endpoint. Following \cite{Lon99,Lon02}, one can associate to $\Ga$ its Bott's function $\Bott: S^1 \to \Z$ which will be a crucial tool throughout this work. It has the following properties:
\begin{itemize}
\item[(a)] (Bott's formula) We have that $\cz(\Ga^k) = \sum_{z^k=1} \Bott(z)$ for every $k \in \N$. In particular, the mean index of $\Ga$ satisfies
\[
\mi(\Ga) = \int_{S^1} \Bott(z)\,dz,
\]
where the total measure of the circle is normalized to be equal to one.

\item[(b)] If $\Ga = \Ga_1 \oplus \Ga_2$ then $\Bott = \Bott_1 + \Bott_2$ where $\Bott_i$ is the Bott's function associated to $\Ga_i$ for $i=1,2$.

\item[(c)] If $\Ga_1$ and $\Ga_2$ are homotopic with fixed endpoints then $\Bott_1=\Bott_2$.

\item[(d)] The discontinuity points of $\Bott$ are contained in $\sigma(P) \cap S^1$, where $\sigma(P)$ is the spectrum of $P$.

\item[(e)] $\Bott(z)=\Bott(\bar z)$ for every $z \in S^1$.

\item[(f)] The \emph{splitting numbers} $S^\pm_z(P) := \lim_{\ep\to 0^+} \Bott(e^{\pm \sqrt{-1}\epsilon}z)-\Bott(z)$ depend only on $P$ and satisfy, for every $z \in S^1$,
\begin{equation}
\label{eq:splitting1}
S^\pm_z(P)=S^\mp_{\bar z}(P),
\end{equation}
\begin{equation}
\label{eq:bott2}
S^\pm_z(P^k)=\sum_{w^k=z} S^\pm_w(P)
\end{equation}
for every $k \in \N$, and
\begin{equation}
\label{eq:bound splitting}
0 \leq S^\pm_z(P) \leq \nu_z(P) \leq \eta_z(P),
\end{equation}
where $\nu_z(P)$ and $\eta_z(P)$ are the geometric and algebraic multiplicities of $z$ respectively if $z \in \sigma(P) \cap S^1$ and zero otherwise. Moreover,
\begin{equation}
\label{eq:bound nullity-splitting}
\nu_{\pm 1}(P) - S^\pm_{\pm 1}(P) \leq \frac{\eta_{\pm 1}(P)}{2}
\end{equation}
and
\begin{equation}
\label{eq:bott via splitting}
\Bott(e^{\sqrt{-1}\theta}) = \Bott(1) + S^+_1(P) + \sum_{\phi \in (0,\theta)} (S^+_{e^{\sqrt{-1}\phi}}(P)-S^-_{e^{\sqrt{-1}\phi}}(P)) - S^-_{e^{\sqrt{-1}\theta}}(P)
\end{equation}
for every $\theta \in [0,2\pi)$. (Note that the sum above makes sense since $S^\pm_z(P)\neq 0$ only for finitely many points $z \in S^1$.)

\item[(g)] $\Bott_\Ga(z)$ is lower semicontinuous with respect to $\Ga$ in the $C^0$-topology. More precisely, let $\P([0,T],\Sp(2n))$ be the set of continuous paths in $\Sp(2n)$ starting at the identity endowed with the $C^0$-topology. Then, for a fixed $z \in S^1$, the map
\[
\P([0,T],\Sp(2n)) \to \Z
\]
that sends $\Ga$ to $\Bott_\Ga(z)$, where $\Bott_\Ga$ denotes the Bott's function associated to $\Ga$, is lower semicontinuous, that is,
\[
\Bott_\Ga(z) = \sup_U \inf_{\Ga' \in U} \Bott_{\Ga'}(z),
\]
where the supremum runs over all $C^0$-neighborhoods $U$ of $\Ga$ in $\P([0,T],\Sp(2n))$.
\end{itemize}

We refer to \cite{Lon02} for a proof of these properties. It is easy to see that the geometric multiplicity satisfies the relation
\begin{equation}
\label{eq:bott3}
\nu_z(P^k)=\sum_{w^k=z} \nu_w(P).
\end{equation}
Moreover, it follows from \cite[Theorem 2.2]{LZ} that if $\cz(\Ga)\geq n$ then 
\begin{equation}
\label{eq:index iteration}
\cz(\Ga^k) + \nu(\Ga^k) \leq \cz(\Ga^{k+1})\quad \forall k \in \N,
\end{equation}
where $\nu(\Ga^k)$ is the geometric multiplicity of the eigenvalue one of the endpoint of $\Ga^k$.

In the proof of Theorem \ref{thm:example open} we will need the following definition. The \emph{upper semicontinuous Bott's function} associated to $\Ga$ is defined as
\begin{equation}
\label{eq:Bott+}
\Bott^+_\Ga(z) = \inf_U \sup_{\Ga' \in U} \Bott_{\Ga'}(z),
\end{equation}
where the infimum runs over all $C^0$-neighborhoods $U$ of $\Ga$ in $\P([0,T],\Sp(2n))$. Note that, for any $z \in S^1$, $\Bott^+_{\Ga}(z)$ is upper semicontinuous with respect to $\Ga$ in the $C^0$-topology. It is easy to see that
\begin{equation}
\label{eq:Bott x Bott+}
\Bott^+_\Ga(z) = -\Bott_{\Ga^{-1}}(z)
\end{equation}
for every $z \in S^1$.

\subsection{A comparison result}
In the proof of Theorem \ref{thm:example open} we will need the following comparison result which has intrinsic interest.

\begin{theorem}
\label{thm:comparison}
Let $\Ga_i: [0,T] \to \Sp(2n)$ ($i=1,2$) be two symplectic paths starting at the identity and satisfying the differential equation
\[
\frac{d}{dt}\Ga_i(t)= JA_i(t)\Ga_i(t),
\]
where $A_i(t)$ is a path of symmetric matrices. Suppose that $A_1(t) \geq A_2(t)$ for every $t$ and let $\Bott_i$ be the Bott's function associated to $\Ga_i$. Then
\[
\Bott_1(z) \geq \Bott_2(z)
\]
for every $z \in S^1$.

\begin{proof}
In what follows, when $i$ or $j$ is not specified, it is any element in $\{1,2\}$. Consider the complexified path $\cGa_i : [0,T] \to \Sp(2n,\C)$. Following \cite{LZ-SF1,LZ-SF2}, given $z \in S^1 \subset \C$, one can analytically associate to $\cGa_i$ an index using the spectral flow. Let 
$$
E_{T,z}=\{x \in W^{1,2}([0,T],\C^{2n});\,x(T)=zx(0)\}\textrm{ and }L_T=L^2([0,T],\C^{2n}).
$$
Consider the family of operators $L^s_{z,A_i}: E_{T,z} \to L_T$ given by
\[
L^s_{z,A_i}x(t) = -J\dot x(t) - sA_i(t)x(t)
\]
for $s \in [0,1]$ and $t \in [0,T]$, where $J$ is the standard complex structure. Each $L^s_{z,A_i}$ is a self-adjoint Fredholm operator. Therefore, we have, for a fixed $z$ and $A_i$, the corresponding spectral flow
\[
\sf(L^s_{z,A_i}) \in \Z,
\]
see, for instance, \cite{APS,LZ-SF1,RS,SW} and references therein. The $z$-index of $\cGa_i$, as defined in \cite[Definition 2.3]{LZ-SF2} and  \cite[Definition 2.8]{LZ-SF1}, is given by
\[
i_z(\cGa_i) = -\sf(L^s_{z,A_i}).
\]
It is related to the Bott's function $\Bott_i$ of $\Ga_i$ in the following way (see \cite[Corollary 2.1]{LZ-SF2}):
\begin{equation}
\label{eq:i_z x Bott}
\Bott_i(z) - i_z(\cGa_i) =
\begin{cases}
-n&\text{if}\ z=1, \\
0&\text{otherwise}.
\end{cases}
\end{equation}
The spectral flow has the following comparison property. Note that $L^s_{z,A_i}x=-J\dot x + C_{z,A_i}^sx$, where $C_{z,A_i}^s$ is the compact self-adjoint operator $C_{z,A_i}^s: E_{T,z} \to L_T$ given by $C_{z,A_i}^sx=-sA_ix$. We say that $C_{z,A_i}^s \geq C_{z,A_j}^s$ if
\[
\lg (C_{z,A_i}^s-C_{z,A_j}^s)x,x \rg_{L^2} \geq 0
\]
for every $x \in E_{T,z}$. It turns out that if $C_{z,A_i}^1 \geq C_{z,A_j}^1$ then
\[
\sf(L^s_{z,A_i}) \geq \sf(L^s_{z,A_j}),
\]
see, for instance, \cite[Theorem 3.9]{SW}. But if $A_1(t) \geq A_2(t)$ for every $t$ then $C_{z,A_2}^s \geq C_{z,A_1}^s$ for every $z$ and $s$. Therefore,
\[
i_z(\cGa_1) \geq i_z(\cGa_2).
\]
Consequently, the result follows from \eqref{eq:i_z x Bott}.
\end{proof}
\end{theorem}

\section{Proof of Theorem \ref{thm:multiplicity}}
\label{sec:pf-multiplicity}

We will use several ideas from \cite{LLZ}. Suppose that $\alpha$ has finitely many simple closed orbits with positive mean index and write
\[
\P^+_s=\{\ga_1,\dots,\ga_{r_1}\}
\]
and
\[
\P^+_{ns}= \bigcup_{i=1}^{r_2} \bigcup_{g \in G} \{g\ga_{r_1+i}\},
\]
where we are identifying $g\ga_{r_1+i}$ with $\ga_{r_1+i}$ whenever $g\ga_{r_1+i}(\R)=\ga_{r_1+i}(\R)$. Clearly, $\cz(\ga_i)=\cz(g\ga_i)$ and $\nu(\ga_i)=\nu(g\ga_i)$ for every $i$ and $g \in G$. Note that $\#\P^+_{ns} \geq 2r_2$ because for every non-symmetric closed orbit $\ga$ there exists some $g\in G$ such that $g\ga \neq \ga$.

Applying Theorem \ref{thm:IRT} to the linearized flows along $\ga_1,\dots,\ga_{r_1+r_2},\ga^2_{r_1+1},\dots,\ga^2_{r_1+r_2}$ we get even positive numbers $d,k_1,\dots,k_{r_1+2r_2}$, which can be chosen to be a multiple of an arbitrarily large number, such that
\begin{equation}
\label{eq:irt1}
\cz(\ga_i^{k_i+1})=d+\cz(\ga_i),
\end{equation}
\begin{align}
\label{eq:irt2}
\cz(\ga_i^{k_i-1})&=d+\cz(\ga_i^{-1})+b_+(\ga_i)-b_-(\ga_i)\nonumber\\
&=d-(\cz(\ga_i)+b_-(\ga_i)-b_+(\ga_i)) - \nu(\ga_i^{k_i-1}),
\end{align}
\begin{equation}
\label{eq:irt3}
d-n \leq \cz(\ga_i^{k_i}) \leq \cz(\ga_i^{k_i})+\nu(\ga_i^{k_i}) \leq d+n
\end{equation}
for every $i \in \{1,\dots,r_1+r_2\}$ and
\begin{equation}
\label{eq:irt4}
\cz(\ga_{i-r_2}^{2(k_i+1)})=d+\cz(\ga_{i-r_2}^2),
\end{equation}
\begin{align}
\label{eq:irt5}
\cz(\ga_{i-r_2}^{2(k_i-1)})&=d+\cz(\ga_{i-r_2}^{-2})+b_+(\ga_{i-r_2}^2)-b_-(\ga_{i-r_2}^2)\nonumber\\
&=d-(\cz(\ga_{i-r_2}^2)+b_-(\ga_{i-r_2}^2)-b_+(\ga_{i-r_2}^2)) - \nu(\ga_{i-r_2}^{2(k_i-1)}),
\end{align}
\begin{equation}
\label{eq:irt6}
d-n \leq \cz(\ga_{i-r_2}^{2k_i}) \leq \cz(\ga_{i-r_2}^{2k_i})+\nu(\ga_{i-r_2}^{2k_i}) \leq d+n
\end{equation}
for every $i \in \{r_1+r_2+1,\dots,r_1+2r_2\}$, where in equations \eqref{eq:irt3} and \eqref{eq:irt6} we used that
\[
\mi(\ga_i^\ell) - n \leq \cz(\ga_i^\ell) \leq \cz(\ga_i^\ell) + \nu(\ga_i^\ell) \leq \mi(\ga_i^\ell) + n
\]
for every $i \in \{1,\dots,r_1+r_2\}$ and $\ell \in \N$ and in equations \eqref{eq:irt2} and  \eqref{eq:irt5} we used that
\[
\cz(\ga_i^{-\ell})=-(\cz(\ga_i^\ell)+\nu(\ga_i^\ell))\ \ \forall i \in \{1,\dots,r_1+r_2\}\ \text{and}\  \forall\ell \in \N
\]
and that $k_1,\dots,k_{r_1+2r_2}$ can be chosen such that
\[
\nu(\ga_i)=\nu(\ga_i^{k_i-1}) \ \ \forall i \in \{1,\dots,r_1+r_2\}
\]
and
\[
\nu(\ga_{i-r_2}^2)=\nu(\ga_{i-r_2}^{2(k_i-1)}) \ \ \forall i \in \{r_1+r_2+1,\dots,r_1+2r_2\}.
\]
(This last assertion follows from the fact that we can make all $k_1,\dots,k_{r_1+2r_2}$ divisible by any natural number.)

We claim that
\begin{equation}
\label{eq:claim}
k_i = 2k_{i+r_2}\ \ \forall i \in \{r_1+1,\dots,r_1+r_2\}.
\end{equation}
As a matter of fact, by our hypotheses, \eqref{eq:irt3} and \eqref{eq:irt5}, we have, for every $i \in \{r_1+1,\dots,r_1+r_2\}$,
\begin{align*}
\cz(\ga_i^{k_i}) & \geq d-n\nonumber\\
& > d - (\cz(\ga_{i}^2)+b_-(\ga_{i}^2)-b_+(\ga_{i}^2))\nonumber\\
& = \cz(\ga_i^{2(k_{i+r_2}-1)}) + \nu(\ga_i^{2(k_{i+r_2}-1)})\nonumber\\
& \geq \cz(\ga_i^{2k_{i+r_2}-2}).
\end{align*}
Since $\cz(\ga)\geq n+1$ for every $\ga \in \P^+$, we have, from \eqref{eq:index iteration}, that the function $m \mapsto \cz(\ga^m)$ is non-decreasing. Therefore,
\[
2k_{i+r_2}-2 < k_i.
\]
On the other hand, by our hypotheses, \eqref{eq:irt3} and  \eqref{eq:irt4},
\begin{align*}
\cz(\ga_i^{2k_{i+r_2}+2}) & = d + \cz(\ga_{i}^2)\nonumber\\
& > d+n\nonumber\\
& \geq \cz(\ga_i^{k_i}),
\end{align*}
implying that
\[
k_i < 2k_{i+r_2}+2.
\]
Consequently,
\[
2k_{i+r_2}-2 < k_i < 2k_{i+r_2}+2.
\]
Since $k_i$ is an even number, we conclude \eqref{eq:claim}.

Now, consider the \emph{carrier map} $\psi: \N \to \P \times \N$ as defined in \cite[Corollary 3.9]{GG}. It is an injective map such that if $\psi(m)=(\ga_i,j)$ then
\[
\cz(\ga_i^j) \leq n+2m \leq \cz(\ga_i^j) + \nu(\ga_i^j).
\]
Given $s \in \{1,\dots,n+1\}$ let $i(s)$ and $j(s)$ be the (unique) numbers satisfying $(\ga_{i(s)},j(s))=\psi(d/2-s+1)$ so that
\[
\cz(\ga_{i(s)}^{j(s)}) \leq d-2s+n+2 \leq \cz(\ga_{i(s)}^{j(s)}) + \nu(\ga_{i(s)}^{j(s)}).
\]
Note that, since $d$ can be chosen to be a positive multiple of an arbitrarily large number, $\ga_{i(s)} \in \P^+$ for every $s$. By our hypotheses and \eqref{eq:irt1},
\begin{align}
\label{eq:bound1}
\cz(\ga_{i(s)}^{j(s)}) & \leq d-2s+n+2\nonumber\\
& \leq d+n\nonumber\\
& < d+\cz(\ga_{i(s)})\nonumber\\
& = \cz(\ga_{i(s)}^{k_{i(s)}+1}).
\end{align}
When $i(s)\leq r_1$ (i.e. $\ga_{i(s)}$ is symmetric), we have, by \eqref{eq:irt2} and our hypotheses,
\begin{align}
\label{eq:bound2}
\cz(\ga_{i(s)}^{k_{i(s)}-1}) + \nu(\ga_{i(s)}^{k_{i(s)}-1}) & = d - (\cz(\ga_{i(s)})+b_+(\ga_{i(s)})-b_-(\ga_{i(s)}))\nonumber\\
& \leq d-n-1\nonumber\\
& < d-2s+n+2\nonumber\\
& \leq \cz(\ga_{i(s)}^{j(s)}) + \nu(\ga_{i(s)}^{j(s)}).
\end{align}
Thus, \eqref{eq:index iteration}, \eqref{eq:bound1} and \eqref{eq:bound2} imply that
\[
k_{i(s)}-1 < j(s) < k_{i(s)}+1\implies j(s)=k_{i(s)}
\]
whenever $i(s)\leq r_1$. If $r_1<i(s)\leq r_1+r_2$ (i.e. $\ga_{i(s)}$ is not symmetric) then, by \eqref{eq:claim}, $k_{i(s)}=2k_{i(s)+r_2}$. Hence, by \eqref{eq:irt5} and our assumptions,
\begin{align}
\label{eq:bound3}
\cz(\ga_{i(s)}^{k_{i(s)}-2}) + \nu(\ga_{i(s)}^{k_{i(s)}-2}) & = d - (\cz(\ga_{i(s)}^2)+b_+(\ga_{i(s)}^2)-b_-(\ga_{i(s)}^2))\nonumber\\
& \leq d-n-1\nonumber\\
& < d-2s+n+2\nonumber\\
& \leq \cz(\ga_{i(s)}^{j(s)}) + \nu(\ga_{i(s)}^{j(s)}).
\end{align}
It follows from \eqref{eq:index iteration}, \eqref{eq:bound1} and \eqref{eq:bound3} that
\[
k_{i(s)}-2 < j(s) < k_{i(s)}+1\implies j(s) \in \{k_{i(s)}-1,k_{i(s)}\}.
\]
Thus, we have that $j(s)=k_{i(s)}$ if $\ga_{i(s)} \in \P^+_s$ and $j(s) \in \{k_{i(s)}-1,k_{i(s)}\}$ if $\ga_{i(s)} \in \P^+_{ns}$. Since the carrier map is injective, if there exist $s_1\neq s_2$ such that $i(s_1)=i(s_2)$ then $j(s_1)\neq j(s_2)$. But if $i(s_1)=i(s_2) \leq r_1$ then $k_{i(s_1)}=k_{i(s_2)}\implies j(s_1)=j(s_2)$, a contradiction. If $i(s_1)=i(s_2) > r_1$ then $\{j(s_1),j(s_2)\}=\{k_{i(s_1)}-1,k_{i(s_1)}\}$. Hence,
\[
\#\{s \in \{1,\dots,n+1\}; j(s)\leq r_1\} \leq r_1
\]
and
\[
\#\{s \in \{1,\dots,n+1\}; j(s) > r_1\} \leq 2r_2,
\]
implying that
\[
\#\P^+ = \#\P^+_s + \#\P^+_{ns} \geq r_1+2r_2 \geq n+1.
\]

\section{Proof of Theorem \ref{thm:c=>sdc1}}
\label{sec:c=>sdc}

\subsection{Idea of the proof}
Let us first describe the idea of the proof. Let $H: \R^{2n+2}\setminus\{0\} \to \R$ be a homogeneous of degree two Hamiltonian such that $\Sigma_\alpha=H^{-1}(1)$. Since $\alpha$ is convex, so is $H$. If $\alpha$ is invariant under the antipodal map then clearly $H$ is invariant under the antipodal map as well.

Given a closed orbit $\ga: [0,T] \to S^{2n+1}$ of $\alpha$, let $\Ga$ be the linearized Hamiltonian flow of $H$ along $\ga$ seen as closed Hamiltonian orbit on $\Sigma_\alpha$. We have that $\cz(\ga)=\cz(\Ga)+1$ and $b_\pm(\ga)=b_\pm(\Ga(T))$ (see Proposition \ref{prop:ReebHam}). Since $H$ is convex, it is well known that $\cz(\Ga)\geq n+1$ and therefore $\cz(\ga)\geq n+2$.

Now, suppose that $\ga$ is symmetric. The fact that $H$ is invariant under the antipodal map implies that $\Ga$ is the second iterate of the linearized Hamiltonian flow along half the orbit $\Ga|_{[0,T/2]}$ (see Proposition \ref{prop:2nd_iterate}). By the convexity of $H$, $\cz(\Ga|_{[0,T/2]}) \geq n+1$ and it turns out that if a symplectic path $\Phi: [0,1] \to \Sp(2n+2)$ satisfies $\cz(\Phi)\geq n+1$ then $\cz(\Phi^2) + b_-(\Phi^2(2)) - b_+(\Phi^2(1)) \geq n+1$ (see Proposition \ref{prop:sdc_2nd_iterate}).

\subsection{Proof of the Theorem}
Given a periodic orbit $\ga$ of $\alpha$, let $\Ga_\ga: [0,T] \to \Sp(2n)$ be the path given by the linearized Reeb flow along $\ga$ (using a trivialization of the contact structure over a capping disk) and $\Ga: [0,T] \to \Sp(2n+2)$ the path given by the linearized Hamiltonian flow of $H$ along $\ga$ (using the constant trivialization of $T\R^{2n+2}$) seen as a periodic orbit of $H$ on $\Sigma_\alpha$. Denote by $\Bott_\ga$ and $\Bott$ the Bott's functions associated to $\Ga_\ga$ and $\Ga$ respectively.

\begin{proposition}
\label{prop:ReebHam}
We have that $\Bott_\ga(1)=\Bott(1)+1$ and $\Bott_\ga(z)=\Bott(z)$ for every $z\neq 1$. Moreover, $b_\pm(\Ga(T))=b_\pm(\Ga_\ga(T))$.
\end{proposition}

\begin{proof}
Let $\xi$ be the contact structure on $\Sigma_\al$ and $\xi^\om$ its symplectic orthogonal with respect to the canonical symplectic form $\om$. Clearly both $\xi$ and $\xi^\om$ are invariant under the linearized Hamiltonian flow of $H$. Consider a capping disk $\sigma: D^2 \to \Sigma_\al$ such that $\sigma|_{\partial D^2}=\ga$. Denote by $\Phi^\xi: \sigma^*\xi \to D^2 \times \R^{2n}$ and $\Phi^{\xi^\om}: \sigma^*\xi^\om \to D^2 \times \R^{2}$ the unique (up to homotopy) trivializations of the pullbacks of $\xi$ and $\xi^\om$ by $\sigma$. Fix a symplectic basis $\{e,f\}$ of $\R^2$. Note that $\Phi^{\xi^\om}$ can be chosen such that $\Phi^{\xi^\om}(X_H)=e$ and $\Phi^{\xi^\om}(Y)=f$, where $X_H$ is the Hamiltonian vector field of $H$ and $Y(x)=x$ (note that $\{X_H(\sigma(x)),Y(\sigma(x))\}$ is a symplectic basis of $\sigma^*\xi^\om(x)$ for every $x \in D^2$). Indeed, let $A: D^2 \to \Sp(2)$ be the map that associates to $x \in D^2$ the unique symplectic map that sends $\Phi^{\xi^\om}(X_H(\sigma(x)))$ to $e$ and $\Phi^{\xi^\om}(Y(\sigma(x)))$ to $f$. Then $\bar A \circ \Phi^{\xi^\om}$ gives the desired trivialization, where $\bar A: D^2 \times \R^2 \to D^2 \times \R^2$ is given by $\bar A(x,v)=(x,A(x)v)$.

We have that $\Phi:=\Phi^\xi \oplus \Phi^{\xi^\om}$ gives a trivialization of $\sigma^*T\R^{2n+2}$. Let $\Ga_\Phi: [0,T] \to \Sp(2n+2)$ be the symplectic path given by the linearized Hamiltonian flow of $H$ along $\ga$ using $\Phi$. By construction, we can write $\Ga_\Phi=\Ga_\Phi^{\xi} \oplus \Ga_\Phi^{\xi^\om}$, where $\Ga_\Phi^{\xi}$ and $\Ga_\Phi^{\xi^\om}$ are given by the linearized Hamiltonian flow of $H$ restricted to $\xi$ and $\xi^\om$ respectively. Let $\Bott_\Phi$, $\Bott^\xi_\Phi$ and $\Bott^{\xi^\om}_\Phi$ be the Bott's function associated to $\Ga_\Phi$, $\Ga_\Phi^{\xi}$ and $\Ga_\Phi^{\xi^\om}$ respectively. By the additivity property of the Bott's function,
\begin{equation}
\label{eq:Bott_split}
\Bott_\Phi(z) = \Bott^\xi_\Phi(z) + \Bott^{\xi^\om}_\Phi(z),
\end{equation}
for every $z \in S^1$.  By construction, the path $\Phi^{\xi^\om}$ is constant equal to the identity ($H$ is homogenous of degree two and therefore its linearized Hamiltonian flow preserves both $X_H$ and $Y$) and consequently
\[
\Bott^{\xi^\om}_\Phi(z) =
\begin{cases}
-1&\text{if}\ z=1 ,\\
0&\text{otherwise}.
\end{cases}
\]
Thus, by \eqref{eq:Bott_split},
\begin{equation}
\label{eq:Bott_Phi}
\Bott_\Phi(z) =
\begin{cases}
\Bott^\xi_\Phi(z)-1&\text{if}\ z=1, \\
\Bott^\xi_\Phi(z)&\text{otherwise}.
\end{cases}
\end{equation}
Now, note that $\Phi$ is homotopic to the usual (global) trivialization of $T\R^{2n+2}$ because both are defined over the whole capping disk. Hence,
\begin{equation}
\label{eq:Bott}
\Bott(z) = \Bott_\Phi(z)
\end{equation}
for every $z$. Moreover, by the construction of $\Phi^{\xi}$,
\begin{equation}
\label{eq:Bott_ga}
\Bott_\ga(z) = \Bott^\xi_\Phi(z)
\end{equation} 
for all $z$. Thus, the first assertion of the proposition follows from \eqref{eq:Bott_Phi}, \eqref{eq:Bott} and \eqref{eq:Bott_ga}. The last one is a consequence of the equality
\[
b_\pm(\Ga(T)) = b_\pm(\Ga_\Phi(T)) = b_\pm(\Ga^\xi_\Phi(T)) + b_\pm(\Ga^{\xi^\om}_\Phi(T)) = b_\pm(\Ga_\ga(T)) + b_\pm(\Ga^{\xi^\om}_\Phi(T))
\]
and the fact that $b_\pm(\Ga^{\xi^\om}_\Phi(T))=0$ because $\Ga^{\xi^\om}_\Phi(T)$ is the identity.
\end{proof}

Now, we need the following algebraic lemma. To simplify notation, we will drop the subscript in $\nu_1$.

\begin{lemma}
\label{lemma:linear}
Let $W$ be a symplectic vector space and $P: W \to W$ a symplectic linear map. Then $b_-(P) - b_+(P)=2S^+_1(P)-\nu(P)$.
\end{lemma}

\begin{proof}
Let $V \subset W$ be the $P$-invariant symplectic subspace whose complexification is the generalized eigenspace of the eigenvalue one. We have that $V$ can be symplectically decomposed into a sum
\[
V = V_\Id \oplus \underbrace{V_0 \oplus \dots \oplus V_0}_{b_0(P)\ \text{times}} \oplus \underbrace{V_+ \oplus \dots \oplus V_+}_{b_+(P)\ \text{times}} \oplus \underbrace{V_- \oplus \dots \oplus V_-}_{b_-(P)\ \text{times}},
\]
where each term is $P$-invariant and $P|_{V_\Id}=\Id$, $P|_{V_0}=P_0:=\exp(JQ_0)$ and $P|_{V_\pm}=P_\pm:=\exp(JQ_\pm)$ up to a (symplectic) change of coordinates, where $Q_0$ and $Q_\pm$ are the quadratic forms defined in Section \ref{sec:sdc} and discussed in \cite[Section 4.1.3]{GG}. Thus,
\begin{equation}
\label{eq:spliting}
S^+_1(P) = S^+_1(P|_{V_\Id}) + b_0(P)S^+_1(P_0) + b_+(P)S^+_1(P_+) + b_-(P)S^+_1(P_-)
\end{equation}
and
\begin{equation}
\label{eq:nullity}
\nu(P) = \nu(P|_{V_\Id}) + b_0(P)\nu(P_0) + b_+(P)\nu(P_+) + b_-(P)\nu(P_-).
\end{equation}
The first terms in the right hand side of the two previous equations satisfy
\begin{equation*}
2S^+_1(P|_{V_\Id}) = \nu(P|_{V_\Id}) = \dim V_\Id.
\end{equation*}
Clearly,
\begin{equation*}
\nu(P_0)=\nu_0(Q_0)=2\ \text{and}\ \nu(P_\pm)=\nu_0(Q_\pm)=1.
\end{equation*}
To compute $S^+_1(P_*)$ for $* \in \{0,\pm\}$ consider the symplectic path $\Psi_*(t)=\exp(tJQ_*)$ for $t \in [0,1]$. Since the spectrum $\sigma(\Psi_*(t))$ of $\Psi_*(t)$ is constant equal to $\{1\}$ for every $t$ we have that $\mi(\Psi_*)=0$. Let $\Bott_*: S^1 \to \Z$ be Bott's function associated to $\Psi_*$. The fact that $\sigma(P_*)=\sigma(\Psi_*(1))=\{1\}$ implies that $\Bott_*$ is constant on $S^1\setminus\{1\}$. But
\[
\int_{S^1} \Bott_*(z)\,dz = \mi(\Psi_*) = 0
\]
implying that $\Bott_*(z)=0$ for every $z\neq 1$. Since $\Bott_*(1)=\cz(\Psi_*)$ we conclude that
\begin{equation*}
S^+_1(P_*) = -\mu(\Psi_*).
\end{equation*}
Therefore,
\begin{equation*}
S^+_1(P_0) = -\frac 12(\sign(Q_0)-\nu_0(Q_0)) = -\frac 12(0-2)=1
\end{equation*}
and
\begin{equation*}
S^+_1(P_\pm) = -\frac 12(\sign(Q_\pm)-\nu_0(Q_\pm)) = -\frac 12(\pm 1-1)=
\begin{cases}
0\ \ \text{for}\ P_+ ,\\
1\ \ \text{for}\ P_-.
\end{cases}
\end{equation*}
Consequently, it follows from \eqref{eq:spliting} and \eqref{eq:nullity} that
\[
2S^+_1(P) - \nu(P) = b_-(P) - b_+(P).
\]
\end{proof}

It is well known that if $\alpha$ is convex then it is dynamically convex, that is, $\mu(\ga) \geq n+2$ for every periodic orbit $\ga$ of $\alpha$. (It follows from Proposition \ref{prop:ReebHam} and the fact that $\mu(\Ga) \geq n+1$ whenever $\alpha$ is convex.) Thus, by the previous two results, in order to prove Theorem \ref{thm:c=>sdc1}, we need to show that 
\[
\cz(\Ga) + 2S^+_1(\Ga(T)) - \nu(\Ga(T)) \geq n+1
\]
for every symmetric closed orbit $\ga$ of $\alpha$, where $\Ga$ is the linearized Hamiltonian flow of $H$ along $\ga$. The proof is based on the following observation.

\begin{proposition}
\label{prop:2nd_iterate}
Suppose that $m=2$. If $\ga$ is symmetric then $\Ga$ is the second iterate of the path $\Ga|_{[0,T/2]}$.
\end{proposition}

\begin{proof}
The path $\Ga$ satisfies the equation
\[
\frac{d}{dt}\Ga(t)=Jd^2H(\ga(t))\Ga(t).
\]
We claim that the Hessian $d^2H(\ga(t))$ is $T/2$-periodic. Indeed,
since $H$ is antipodally symmetric, we have that $dH(-x)=-dH(x)$, and hence
\[
d^2H(-x)=d^2H(x)
\]
for every $x \in \R^{2n+2}\setminus\{0\}$. The symmetry of $\ga$ means that $\ga(t+T/2)=-\ga(t)$ for every $t$ and consequently
\[
d^2H(\ga(t+T/2))=d^2H(\ga(t))
\]
for every $t$.
\end{proof}

Now Theorem \ref{thm:c=>sdc1} follows from Lemma \ref{lemma:linear}, Propositions \ref{prop:ReebHam} and \ref{prop:2nd_iterate} and the following result proved in \cite[Lemma 4.1]{LLZ} and \cite[Lemma 15.6.3]{Lon02}. For the sake of completeness, we will provide a detailed argument.

\begin{proposition}[\cite{LLZ,Lon02}]
\label{prop:sdc_2nd_iterate}
Let $\Ga: [0,1] \to \Sp(2n+2)$ be a symplectic path starting at the identity. If $\cz(\Ga)\geq n+1$ then
\[
\cz(\Ga^2) + 2S^+_1(\Ga^2(2)) - \nu(\Ga^2(2)) \geq n+1.
\]
\end{proposition}

\begin{proof}
Let $P=\Ga(1)$ and $\Bott$ be the Bott's function associated to $\Ga$. By \eqref{eq:splitting1}, \eqref{eq:bound splitting} and \eqref{eq:bound nullity-splitting} we have that
\begin{align}
\label{eq:bound}
& \sum_{\theta \in (0,\pi)} S^-_{e^{\sqrt{-1}\theta}}(P) + (\nu_1(P)-S^+_1(P)) + (\nu_{-1}(P)-S^+_{-1}(P)) \nonumber\\
& \leq \sum_{\theta \in (0,\pi)} \eta_{e^{\sqrt{-1}\theta}}(P) + \frac{\eta_1(P)}{2} + \frac{\eta_{-1}(P)}{2} \nonumber\\
& \leq n+1.
\end{align}
Thus, omitting the dependence of $S^\pm$ and $\nu$ on $P$, we arrive at
\begin{align*}
& \cz(\Ga^2) + 2S^+_1(\Ga^2(2)) - \nu(\Ga^2(2)) \\
& = \Bott(1) + \Bott(-1) + 2(S^+_1+S^+_{-1}) - (\nu_1+\nu_{-1}) \\
& = 2\Bott(1) + \bigg(S^+_1 + \sum_{\theta \in (0,\pi)} (S^+_{e^{\sqrt{-1}\theta}}-S^-_{e^{\sqrt{-1}\theta}}) - S^+_{-1}\bigg) + 2(S^+_1+S^+_{-1}) - (\nu_1+\nu_{-1}) \\
& = 2\Bott(1) + 2S^+_1 + \sum_{\theta \in (0,\pi)} (S^+_{e^{\sqrt{-1}\theta}}-S^-_{e^{\sqrt{-1}\theta}}) - (\nu_{1}-S^+_{1}) - (\nu_{-1}-S^+_{-1}) \\
& \geq 2\Bott(1) - \bigg(\sum_{\theta \in (0,\pi)} S^-_{e^{\sqrt{-1}\theta}} + (\nu_{1}-S^+_{1}) + (\nu_{-1}-S^+_{-1})\bigg) \\
& \geq n+1,
\end{align*}
where the first equality follows from Bott's formula, \eqref{eq:bott2} and \eqref{eq:bott3}, the second equality holds by \eqref{eq:splitting1} and \eqref{eq:bott via splitting}, the first inequality follows from the fact that the splitting numbers are non-negative and the last inequality is a consequence of \eqref{eq:bound} and the hypothesis of the lemma.
\end{proof}

\section{Proof of Theorem \ref{thm:example}}
\label{sec:proof main result}

Let $F: \R^2 \to \R$ be defined as $F(q,p)=\pi(q^2+p^2)$. Given $\ep>0$ consider the Hamiltonian $G: \R^{2n} \to \R$ given by
\[
G(q_1,p_1,\dots,q_n,p_n) = - \big( F(q_1,p_1) + F(q_2,p_2)^2 + \sum_{i=3}^n \ep F(q_i,p_i) \big).
\]
Let $f: \R \to \R$  be a smooth function such that
\begin{itemize}
\item $f(0)=1$, $f'(0)=1$;
\item $f'(r) \in (0,1)$ and $f''(r)>0$ for every $r\in (-r_0,0)$ and some $r_0>0$;
\item $f(r)=C$ for every $r\leq -r_0$, where $C$ is a constant bigger than $1/2$.
\end{itemize}
Define the Hamiltonian
\[
H=f \circ G,
\]
where the constants $\ep$ and $r_0$ will be properly chosen. 

\begin{remark}
\label{rmk:H C1 close to 1}
Note that, choosing $\ep<1$ and $r_0$ very small and $C$ close enough to 1, we can make $H$ arbitrarily uniformly $C^1$-close to the constant function equal to one.
\end{remark}

Consider the standard contact sphere $(S^{2n+1},\xi)$ as the prequantization circle bundle of $\CP^n$ with connection form $\beta$ and projection $\pi: S^{2n+1} \to \CP^n$. Take $x_0 \in \CP^n$ and a neighborhood $U$ of $x_0$ with Darboux coordinates $(q_1,p_1,\dots,q_n,p_n)$ identifying $x_0$ with the origin. Taking $r_0$ sufficiently small and viewing $H$ as an Hamiltonian on $U$, extend $H$ to $\CP^n$ setting $H|_{\CP^n\setminus U}\equiv C$. Define the contact form
\[
\alpha=\beta/\hat H,
\]
where $\hat H=H\circ\pi$. The Reeb vector field of $\alpha$ is given by
\begin{equation}
\label{eq:Reeb}
R_\alpha=\hat HR_{\beta}+ \hat X_{H},
\end{equation}
where $R_{\beta}$ is the Reeb vector field of $\beta$ (whose flow generates a freee circle action with minimal period one) and $\hat X_{H}$ is the horizontal lift of the Hamiltonian vector field of $H$. By the construction of $H$, clearly $W:=\pi^{-1}(U)$ is invariant under the Reeb flow of $\alpha$ and outside $W$ the Reeb vector field of $\alpha$ is a constant multiple of $R_{\beta}$.

Let $\ga_0$ be the simple closed orbit of $R_\alpha$ over $x_0$. The next lemma establishes that $\alpha$ is dynamically convex except at $\ga_0$. 

\begin{lemma}
\label{lemma:dynconvex}
Every periodic orbit $\ga$ of $\alpha$ distinct from $\ga_0$ satisfies $\mu(\ga)\geq n+2$. Moreover, $\mu(\ga_0)=n$.
\end{lemma}

This lemma and other lemmas below, which the proof of the theorem relies on and are not readily availabe in the literature, are proved in Section \ref{sec:technical}.

Now, fix a section of the determinant line bundle $\bigwedge^n_{\C}\xi$. Using this section, we can define the mean index $\mi(\eta)$ for all finite segments $\eta$ of Reeb orbits, not necessarily closed, for any contact form on $(S^{2n+1},\xi)$. This index depends continuously on the initial condition and the contact form (in the $C^2$-topology), and for closed orbits it agrees with the standard mean index defined using trivializations of $\xi$ over capping disks. We say that a contact form $\alpha$ is \emph{index-positive} if there are constants $b>0$ and $c$ such that
\[
\mi(\eta) \geq bT+c
\]
for every Reeb segment $\eta: [0,T] \to M$ of $\alpha$.

\begin{lemma}
\label{lemma:index+}
The contact form $\alpha$ is index-positive.
\end{lemma}

Clearly, $\alpha$ is invariant under the Reeb flow of $\beta$. In particular, it is antipodally symmetric (the antipodal map is given by the time $1/2$ map of the flow of $R_{\beta}$). Let $\balpha$ be the induced form on $\RP^{2n+1}$ and $\bga_0: [0,1/2] \to \RP^{2n+1}$ be the simple closed orbit of $\balpha$ such that its second iterate is the projection of $\ga_0$. Denote by $\bar\beta$ the connection form on $\RP^{2n+1}$ induced by $\beta$ and let $\bar\pi: \RP^{2n+1} \to \CP^n$ be the projection. Clearly,
\begin{equation}
\label{eq:Reeb'}
R_{\balpha}=\hat H'R_{\bar\beta}+ \hat X_{H}',
\end{equation}
where $\hat H'= H\circ\bar\pi$ and $\hat X_{H}'$ is the horizontal lift of the Hamiltonian vector field of $H$ with respect to the connection form $\bar\beta$.

Shrinking $U$ if necessary, consider coordinates $(q_1,p_1,\dots,q_n,p_n,\theta)$ on $\bar W:=(\bar\pi)^{-1}(U) \simeq U \times \R/\frac 12\Z$ such that
\[
\bar\beta|_{\bar W}=\lambda+d\theta,
\]
where $\lambda=\frac 12\sum_{i=1}^n (q_idp_i-p_idq_i)$ is the Liouville form. Consider the section $\Sigma=U \times \{0\}$ transversal to $\bga_0$ and take a possibly smaller section $S=W \times \{0\}$, where $W \subset U$ is an open subset containing $x_0$. Let $P: S \to P(S)$ be the corresponding first return map.

We will also consider the following coordinates on a neighborhood of $\bga_0$.

\begin{lemma}\cite[Lemma 5.2]{HM}
\label{lemma:coords}
There exist a neighborhood $V \simeq U' \times \R/\frac 12\Z$ of $\bga_0$, where $U' \subset \R^{2n}$ is a small neighborhood of the origin, and coordinates $(q_1',p_1',\dots,q_n',p_n',t)$ on $V$ such that
\[
\balpha|_V = \lambda' + H_t dt,
\]
where $\lambda'=\frac 12\sum_{i=1}^n (q_i'dp_i'-p_i'dq_i')$ is the Liouville form and $H_t: U' \to \R$ is a $1/2$-periodic Hamiltonian such that $H_t(0)=1$ and $dH_t(0)=0$.
\end{lemma}

Consider, with respect to the coordinates given by the previous lemma, the section $\Sigma'=U'\times \{0\}$ and let $S'$ be a possibly smaller section $S'=B \times \{0\}$ with $B \subset U'$ an open subset containing the origin. Let $P': S' \to P(S')$ be the corresponding first return map. Shrinking $S'$ if necessary, we have that $P'$ is a well defined symplectic diffeomorphism. The following lemma is well known and therefore we will omit its proof.

\begin{lemma}
\label{lemma:returns}
Shrinking $S$ and $S'$ if necessary, we have that $P$ are $P'$ are symplectically conjugate.
\end{lemma}

Thus, up to a change of coordinates, we can assume that $P'=P$. $P'$ is given by the time $1/2$ map of $H_t$ but $P$ \emph{is not} the time $1/2$ map of $H$; see Section \ref{sec:proof no new orbit}. (Note the difference between the time-dependent Hamiltonian $H_t$ and the autonomous Hamiltonian $H$.)

Now, consider the Hamiltonian
\[
S(q_1',p_1',\dots,q_n',p_n') = \frac 12(p_1'^2+p_2'^2)
\]
whose flow is given by
\[
\vr^S_t(q_1',p_1',\dots,q_n',p_n')=(S^t(q_1',p_1'),S^t(q_2',p_2'),q_3',p_3',\dots,q_n',p_n'),
\]
where $S^t$ is the symplectic shear
\[
S^t = \left(
\begin{matrix}
1 & t \\
0 & 1
\end{matrix} \right).
\]
Take $0< r' < r'' < \sup\{-G(x);\,x\in B\}$ and define  $B'=\{x\in B;\,-G(x)<r'\}$ and $B''=\{x\in B;\,-G(x)<r''\}$. Let $\bump: B \to \R$ be a bump function such that $\bump|_{B'}\equiv 1$ and $\bump|_{B\setminus B''}\equiv 0$. Take a non-decreasing function $\chi: [0,1/2] \to [0,1/2]$ such that $\chi(t)=0$ for every $t \in [0,\delta]$ and $\chi(t)=1/2$ for every $t \in [1/2-\delta,1/2]$ for some small $\delta>0$. Given $\ep'>0$ sufficiently small consider the $C^2$-small perturbation of $H_t$
\[
H^{\ep'}_t(x)=H_t(x) + \ep'\chi'(t)\bump(x)S((\vr^{H_t}_{t})^{-1}(x))
\]
for $t \in [0,1/2]$ and $H^{\ep'}_{t+1/2}=H^{\ep'}_t$ for every $t$. Taking $\ep'$ sufficiently small, consider the contact form $\balpha_{\ep'}$ on $\RP^{2n+1}$ given by
\[
\balpha_{\ep'}|_V=\lambda' + H^{\ep'}_tdt
\]
and $\balpha_{\ep'} = \balpha$ away from $B \times \R/\frac 12 \Z$.

Let $\alpha_{\ep'}$ be the lift of $\balpha_{\ep'}$ to $S^{2n+1}$ and $\ga^{\ep'}_0$ be the simple periodic orbit of $\alpha_{\ep'}$ whose image coincides with that of $\ga_0$.

\begin{lemma}
\label{lemma:no new orbit}
There exist a neighborhood of $\ga_0$, $\tau>0$ and $\ep_0>0$ such that, for every $\ep'<\ep_0$, the orbit $\ga^{\ep'}_0$ is the only periodic orbit of $\alpha_{\ep'}$ entirely contained in this neighborhood with period $T \in (1-\tau,1+\tau)$.
\end{lemma}

\begin{lemma}
\label{lemma:gamma0}
There exists $\ep_0>0$ such that
\[
\mu(\ga^{\ep'}_0) = n+2\quad\text{and}\quad \mu(\ga^{\ep'}_0(1))+b_-(\ga^{\ep'}_0)-b_+(\ga^{\ep'}_0(1))=n
\]
for every $0 < \ep'<\ep_0$.
\end{lemma}

It follows from Lemmas \ref{lemma:dynconvex}, \ref{lemma:index+}, \ref{lemma:no new orbit} and \ref{lemma:gamma0} that $\alpha_{\ep'}$ is dynamically convex for every $\ep'$ sufficiently small. As a matter of fact, if a periodic orbit $\ga$ of $\alpha_{\ep'}$ is close to a periodic orbit of $\alpha$ different from $\ga_0$ then it follows from the lower semicontinuity of the index and Lemma \ref{lemma:dynconvex} that $\mu(\ga)\geq n+2$. If $\ga$ is close to $\ga_0$ then, by Lemma \ref{lemma:no new orbit}, $\ga$ is equal to $\ga^{\ep'}_0$ and we conclude, by Lemma \ref{lemma:gamma0}, that $\mu(\ga)=n+2$. Finally, if $\ga$ is a new orbit of $\alpha_{\ep'}$ (that is, $\ga$ is not close to any periodic orbit of $\alpha$) then its period is very large and the index-positivity of $\alpha$ (Lemma  \ref{lemma:index+}) assures that $\mu(\ga)\geq n+2$.

To conclude the proof of Theorem \ref{thm:multiplicity}, notice that $\alpha_{\ep'}$ is antipodally symmetric and clearly $\ga_0^{\ep'}$ is symmetric. Thus, it follows from Lemma \ref{lemma:gamma0} that $\alpha_{\ep'}$ is not strongly dynamically convex.

\section{Proofs of Technical Lemmas}
\label{sec:technical}

\subsection{Proof of Lemma \ref{lemma:dynconvex}}
\label{sec:dynconvex}

Consider on $W \simeq U \times S^1$ the coordinates $(x,\theta)$, where $x=(q_1,p_1,\dots,q_n,p_n)$ and $\theta$ is the coordinate along the fiber such that, with respect to these coordinates, $x_0$ is the origin and
\[
\beta|_W=\lambda+d\theta,
\]
where $\lambda=\frac 12\sum_{i=1}^n (q_idp_i-p_idq_i)$ is the Liouville form. Clearly, if $\ga$ lies outside $W$ then $\mu(\ga)\geq n+2$. So suppose that the image of $\ga$ is contained in $W$. The Darboux coordinates induce an obvious (constant) trivialization $D: TU \to U \times \R^{2n}$. From this we get a trivialization of $\xi|_W$ given by
\begin{equation}
\label{eq:Phi}
\Phi(v)=\pi_2(D(\pi_*v)),
\end{equation}
where $\pi_2: U \times \R^{2n} \to \R^{2n}$ is the projection onto the second factor. It is clear that
\begin{equation}
\label{eq:index Phi}
\mu(\ga,\Phi)=\mu(\ga_H),
\end{equation}
where $\mu(\ga,\Phi)$ stands for the index of $\ga$ with respect to the trivialization $\Phi$ and $\ga_H=\pi\circ\ga$ is the corresponding orbit of $H$ with the index computed using the trivialization $D$. Let $\mathfrak f$ be the generator of $\pi_1(W)\simeq \Z$ given by the homotopy class of a simple orbit of $R_{\beta}$ contained in $W$. Let $q\in \Z$ be such that $[\ga]=q\mathfrak f$, where $[\ga]$ is the homotopy class of $\ga$ in $W$. It turns out that $q$ is given by the Hamiltonian action of $\ga_H$. Indeed,
\begin{equation}
\label{eq:action}
q=\int_\ga d\theta = \int_0^T \beta(R_\balpha(\ga(t))) - \lambda(X_{H}(\ga_H(t)))\,dt = \int_0^T H(\ga_H(t)) - \lambda(X_{H}(\ga_H(t)))\,dt,
\end{equation}
where $T$ is the period of $\ga_H$. Therefore, the action $A_{H}(\ga_H)$ is an integer number.

Consider a trivialization $\Psi$ of $\ga^*\xi$ induced by a capping disk. The relation between the trivializations $\Phi$ and $\Psi$ are given by the following lemma.

\begin{lemma}
\label{lemma:triv}
We have that
\[
\mu(\ga,\Psi)=\mu(\ga,\Phi)+q(2n+2).
\]
\end{lemma}

\begin{proof}
Choose a simple orbit $\phi(t)$ of $R_{\beta}$ contained in $W$ and let $Q: [0,1] \times S^1 \to W$ be a homotopy between $\ga$ and $\phi^q$. We can extend the trivialization $\Psi$ to $Q^*\xi$ inducing a trivialization of $(\phi^q)^*\xi$. We have that
\[
\mu(\ga,\Psi)-\mu(\ga,\Phi)=\mu(\phi^q,\Psi)-\mu(\phi^q,\Phi).
\]
But an easy computation shows that
\[
\mu(\phi^q,\Psi)-\mu(\phi^q,\Phi)=q(2n+2).
\]
\end{proof}

From now on, \emph{if the trivialization is not explicitly stated we use a trivialization given by a capping disk for $\ga$ and the trivialization $D$ for closed orbits of Hamiltonians on $U$}. Let $U_0=\{x\in U; -G(x)<r_0\}$. Clearly, $U_0$ is invariant under the Hamiltonian flow of $H$ and if the image of $\ga_H$ is not contained in $U_0$ then $\mu(\ga_H)=-n$ which implies, by Lemma \ref{lemma:triv}, that $\mu(\ga)\geq n+2$.

Thus suppose that the image of $\ga_H$ lies in $U_0$. If $\ga_H(t)=0$ for some $t$ then $\ga_H=(\ga_H^0)^k$ for some $k\in \N$, where $\ga_H^0: [0,1] \to U_0$ is the constant solution $\ga_H^0(t)\equiv 0$. A direct computation shows that in this case, due to our choice of $f$, the linearized Hamiltonian flows of $H$ and $G$ along $\ga_H$ coincide and therefore have the same index. But the Hamiltonian flow of $G$ is given by
\begin{equation}
\label{eq:flow G}
\vr^G_t(z_1,\dots,z_n)=(e^{-2\pi \sqrt{-1}t}z_1,e^{-4F(z_2)\pi\sqrt{-1}t}z_2,e^{-2\ep\pi \sqrt{-1}t}z_3,\dots,e^{-2\ep\pi \sqrt{-1}t}z_n),
\end{equation}
where we are identifying $(q_i,p_i)$ with $z_i=q_i+\sqrt{-1}p_i$. Thus, it is clear that if $k=1$ then, choosing $\ep<1$, we have
\[
\mu(\ga_H)=-3-1-(n-2)=-n-2.
\]
This implies, by Lemma \ref{lemma:triv}, that $\mu(\ga)=n$. Moreover, if $k>1$ then
\[
\mu(\ga_H)=
\begin{cases}
-(2k+1)-1-(n-2)(2\lceil k\ep\rceil+1)\quad\text{if}\ k\ep \notin \Z \\
-(2k+1)-1-(n-2)(2\lceil k\ep\rceil-1)\quad\text{otherwise,}
\end{cases}
\]
where $\lceil x\rceil = \min\{k \in \Z; k \geq x\}$. Thus, since $q=k$, using Lemma \ref{lemma:triv} and choosing $\ep$ sufficiently small we conclude that
\[
\mu(\ga)\geq n+2.
\]

Now, let us consider the remaining case, where $\ga_H$ lies inside $U_0$ and $\ga_H(t)\neq 0$ for every $t$. Let $\ga_G(t)=\ga_H(t/f'(G(\ga_H(0))))$ be the corresponding orbit of $G$.

\begin{lemma}
\label{lemmma:indexes HS}
We have that $|\mu(\ga_H)-\mu(\ga_G)|\leq 1$.
\end{lemma}

\begin{proof}
Firstly, notice that it is enough to find some symplectic trivialization $\Phi^t: T_{\ga_G(t)}\R^{2n} \to \R^{2n}$ such that
\[
|\mu(\ga_G,\Phi^t)-\mu(\ga_H,\Phi^{f'(e)t})|\leq 1,
\]
where $e=G(\ga_H(0))$ and the indexes above are the indexes of the symplectic paths defined using the corresponding trivializations. Let $T$ and $T_G=f'(G(\ga_H(0)))T$ be the periods of $\ga_H$ and $\ga_G$ respectively. We claim that there exists a symplectic plane $P$ such that $X_G(x)\in P$ and $d\vr^G_{T_G}(x)P=P$, where $x=\ga_H(0)$. Indeed, write $x=(z_1,\dots,z_n)$ and let $v_i=(0,\dots,0,z_i,0,\dots,0)$, $w_i=(0,\dots,0,\sqrt{-1}z_i,0,\dots,0)$ and $P_i=\tspan\{v_i,X_G(x)\}$. Note that if $z_i \neq 0$ then $P_i$ is a symplectic plane because $X_G(x)=-2\pi\sqrt{-1}(z_1,2F(z_2)z_2,\ep z_3,\dots,\ep z_n)$ and $\om(v_i,w_i)\neq 0$. Moreover, by \eqref{eq:flow G}, $d\vr^G_{T_G}(x)P_i=P_i$ for every $i\neq 2$ and $d\vr^G_{T_G}(x)P_2=P_2$ if $z_i=0$ for every $i\neq 2$. Thus, if $z_2 \neq 0$ and $z_i=0$ for every $i\neq 2$ we take $P=P_2$; if $z_i\neq 0$ for some $i\neq 2$ we take $P=P_i$.

Now, let $S=G^{-1}(e)$. Define
\[
D_1(t)=d\vr^G_{t}(x)P\quad\text{and}\quad D_2(t)=d\vr^G_{t}(x)P^\om=D_1(t)^\om,
\]
where $P^\om$ is the symplectic orthogonal to $P$. Since $X_G(\ga_G(t)) \in D_1(t)$ we have that $D_2(t) \subset T_{\ga_G(t)}S$ for every $t$. By construction, $D_1$ and $D_2$ are both invariant under the linearized Hamiltonian flow of $G$. But $\vr^H_t|_S=\vr^G_{f'(e)t}|_S$ which implies that $D_2$ is also invariant under the linearized Hamiltonian flow of $H$ and therefore the same holds for $D_1$ (since $D_1=D_2^\om$). Moreover, $X_H(\ga_G(t))=f'(e)X_G(\ga_G(t)) \in D_1(t)$ for every $t$.

Take symplectic trivializations $\Phi^t_1: D_1(t) \to \R^2$ and $\Phi^t_2: D_2(t) \to \R^{2n-2}$. We can choose $\Phi_1$ such that $\Phi^t_1(X_G(\ga_G(t)))=v$ for every $t$, where $v$ is a fixed vector in $\R^2$. Let $\Ga^G_1(t)=\Phi^t_1 \circ d\vr^G_{t}(x) \circ (\Phi^0_1)^{-1}$, $\Ga^G_2(t)=\Phi^t_2 \circ d\vr^G_{t}(x) \circ (\Phi^0_2)^{-1}$ and $\Ga^G(t)=\Phi^t \circ d\vr^G_{t}(x) \circ (\Phi^0)^{-1}$ be the corresponding symplectic paths, where $\Phi=\Phi_1\oplus\Phi_2$. Similarly, define $\Ga^H_1(t)=\Phi^{f'(e)t}_1 \circ d\vr^H_{t}(x) \circ (\Phi^0_1)^{-1}$, $\Ga^H_2(t)=\Phi^{f'(e)t}_2 \circ d\vr^H_{t}(x) \circ (\Phi^0_2)^{-1}$ and $\Ga^H(t)=\Phi^{f'(e)t} \circ d\vr^H_{t}(x) \circ (\Phi^0)^{-1}$. Since the Hamiltonian flows of $G$ and $H$ restricted to $S$ differ by a constant reparametrization,
\[
\mu(\Ga^G_2)=\mu(\Ga^H_2).
\]
By our choice of $\Phi_1$, we have that the spectra of $\Ga^G_1(t)$ and $\Ga^H_1(t)$ are equal to $\{1\}$ for every $t$. This implies that
\[
\mu(\Ga^G_1) \in \{0,-1\}\quad\text{and}\quad\mu(\Ga^H_1) \in \{0,-1\}
\]
and consequently $|\mu(\Ga^G_1)-\mu(\Ga^H_1)|\leq 1$. Thus,
\[
|\mu(\Ga^G)-\mu(\Ga^H)|\leq 1
\]
as desired.
\end{proof}

Now, let $k: \R \to \R$ be a smooth function such that $k'(r)>0$ for every $r>0$. Define $\delta: \R \to \Z$ as
\[
\delta(x)=
\begin{cases}
2x +1\quad\text{if}\ x \in \Z ,\\
2\lceil x\rceil -1\quad\text{otherwise,}
\end{cases}
\]
where, as before, $\lceil x\rceil = \min\{k \in \Z; k \geq x\}$. In what follows, recall that $F: \R^2 \to \R$ is the Hamiltonian given by $F(q,p)=\pi(q^2+p^2)$.

\begin{lemma}
\label{lemma:index K}
Let $K=-k\circ F$ with $k$ as above. Given a periodic orbit $\ga_K$ of $K$ with period $T_K$ then
\[
\mu(\ga_K) \geq -\delta(k'(F(\ga_K(0)))T_K).
\]
\end{lemma}

\begin{proof}
The flow of $K$ is given by $\vr^K_t(z)=e^{-2\pi k'(F(z))\sqrt{-1}t}z$, where, as before, we are identifying $(q,p)$ with $z=q+\sqrt{-1}p$. An easy computation shows that if $\ga_K(t)\equiv 0$ then the linearized Hamiltonian flow on $\ga_K$ is given by $e^{-2\pi k'(0)\sqrt{-1}t}z$ and consequently $\mu(\ga_K) = -\delta(k'(0)T_K)$.

If $\ga_K$ is away from the origin, we proceed similarly as in the proof of the previous lemma. Let $\ga_{-F}(t)=\ga_K(t/k'(F(\ga_K(0))))$ be the corresponding periodic orbit of $-F$ with period $T_{-F}=k'(F(\ga_K(0)))T_K$. Take a symplectic trivialization $\Phi^t: T_{\ga_{-F}(t)}\R^2 \to \R^2$  such that $\Phi^t(X_{-F}(\ga_{-F}(t)))=v$ and  $\Phi^t(-\nabla F(\ga_{-F}(t))/\|\nabla F(\ga_{-F}(t))\|)=w$ (here the gradient and the norm are taken with respect to the Euclidean metric) where $\{v,w\}$ is a fixed symplectic basis in $\R^2$. Then clearly the linearized Hamiltonian flow of $-F$ with respect to this trivialization is constant equal to the identity and therefore
\[
\mu(\ga_{-F},\Phi)=-1.
\]
Since $X_K(\ga_K(t))=k'(F(\ga_K(0)))X_{-F}(\ga_K(t))$ is preserved under the linearized Hamiltonian flow of $K$, we see that  the spectrum of the symplectic path
\[
\Ga^K(t)=\Phi^{k'(F(\ga_K(0)))t} \circ d\vr^K_{t}(x) \circ (\Phi^0)^{-1}
\]
is constant and equal to $\{1\}$. Thus,
\[
\mu(\ga_K,\Phi)=\mu(\Ga^K) \in \{-1,0\}.
\]
In particular, we conclude that $\mu(\ga_K,\Phi) \geq \mu(\ga_{-F},\Phi)$. But this implies that
\[
\mu(\ga_K) \geq \mu(\ga_{-F}),
\]
where the indexes above are computed using the canonical trivialization of $\R^2$. Finally, an easy computation shows that
\[
\mu(\ga_{-F})=-\delta(k'(F(\ga_K(0)))T_K).
\]
\end{proof}

Therefore, it follows from \eqref{eq:flow G} and Lemma \ref{lemma:index K} that
\[
\mu(\ga_G) \geq -\delta(T_G) - \delta(2F(\tau_2(\ga_G(0)))T_G) - (n-2)\delta(\ep T_G),
\]
where $\tau_2: \R^{2n} \to \R^2$, given by $\tau_2(q_1,p_1,\dots,q_n,p_n)=(q_2,p_2)$, is the projection onto the second factor. Let $q$ be the integer number such that $[\ga]=q\mathfrak f$ given by \eqref{eq:action}. It is clear from the previous inequality and Lemmas \ref{lemma:triv} and \ref{lemmma:indexes HS} that
\begin{equation}
\label{eq:index bga}
\mu(\ga) \geq q(2n+2) -\delta(T_G) - \delta(2F(\tau_2(\ga_G(0)))T_G) - (n-2)\delta(\ep T_G) -1.
\end{equation}
(Recall that $\mu(\ga)$ is the index of $\ga$ with respect to a trivialization given by a capping disk.) Thus, to prove that $\mu(\ga)\geq n+2$ it is enough to show that
\begin{equation}
\label{eq:dynconvex}
q^{-1}(n+2 +\delta(T_G) + \delta(2F(\tau_2(\ga_G(0)))T_G) + (n-2)\delta(\ep T_G)+1) \leq 2n+2.
\end{equation}
In order to prove this inequality, note first that $q\geq 2$. Indeed, this is a consequence of the following result.

\begin{lemma}
\label{lemma:action}
If $\ga_H$ is a non-constant periodic orbit of $H$ then $A_{H}(\ga_H)>1$.
\end{lemma}

\begin{proof}
Let $T$ be the period of $\ga_H$. We have that
\begin{align*}
A_{H}(\ga_H) & = \int_0^T f(G(\ga_H(t))) - f'(G(\ga_H(t)))\lambda(X_{G}(\ga_H(t)))\,dt \\
& = Tf(G(\ga_H(0))) - f'(G(\ga_H(0)))\int_0^T \lambda(X_{G}(\ga_H(t)))\,dt.
\end{align*}
An easy computation shows that
\[
\lambda(X_{G}(q_1,p_1,\dots,q_n,p_n)) = -\big( F(q_1,p_1)+2F(q_2,p_2)^2+\ep\sum_{i=3}^n F(q_i,p_i) \big).
\]
But the Hamiltonian
\[
G(q_1,p_1,\dots,q_n,p_n) - F(q_2,p_2)^2 = -\big( F(q_1,p_1)+2F(q_2,p_2)^2+ \ep\sum_{i=3}^n F(q_i,p_i) \big)
\]
commutes with $G$ and therefore $\lambda(X_{G}(\ga_H(t)))$ is independent of $t$. Thus,
\[
A_H(\ga_H) = T(f(G)-f'(G)A),
\]
where $A:=\lambda(X_{G}(\ga_H(0)))=G(\ga_H(0))-F(\tau_2(\ga_H(0)))^2$ and, to simplify the notation, we have omitted the dependence of $G$ on $\ga_H(0)$.

Let $T_G=Tf'(G)$ be the period of the corresponding orbit of $G$. Since $A\leq G$ and $f'(G)\geq 0$, we arrive at
\begin{align}
\label{eq:action H}
A_H(\ga_H) & = T(f(G)-f'(G)A) \nonumber\\
& \geq T(f(G)-f'(G)G) \nonumber\\
& = T_G\frac{f(G)-f'(G)G)}{f'(G)} \nonumber\\
& = s(G)T_G,
\end{align}
where $s: (-r_0,0] \to \R$ is given by
\[
s(r)=f(r)/f'(r)-r.
\]
We claim that $s(G)>1$. As a matter of fact, since $\ga_H$ is non-constant, $G(\ga_H(0)) \in (-r_0,0)$. But, by our choice of $f$, $s(0)=1$ and
\[
s'(r) = \frac{-f''(r)f(r)}{f'(r)^2} < 0
\]
for every $r \in (-r_0,0)$. Consequently,
\begin{equation}
\label{eq:bound action H}
A_H(\ga_H) > T_G \geq 1,
\end{equation}
where the last inequality follows from \eqref{eq:flow G}.
\end{proof}

Now notice that
\[
\delta(cT_G) \leq 2c\lceil T_G\rceil +1
\]
for every positive real number $c$. Hence,
\begin{align*}
& q^{-1}(n+2 +\delta(T_G) + \delta(2F(\tau_2(\ga_G(0)))T_G) + (n-2)\delta(\ep T_G)+1) \\
& \leq \frac{2n+3}{q} + \frac{2\lceil T_G\rceil}{q}(1+2F(\tau_2(\ga_G(0)))+(n-2)\ep).
\end{align*}

In order to estimate the last expression, note that, by \eqref{eq:bound action H} and the fact that $q=A_H(\ga_H)$ is an integer,
\[
\frac{\lceil T_G\rceil}{q} \leq 1.
\]
Hence, since $q\geq 2$,
\[
\frac{2n+3}{q} + \frac{2\lceil T_G\rceil}{q}(1+2F(\tau_2(\ga_G(0)))+(n-2)\ep) \leq n+\frac{7}{2}+2(2F(\tau_2(\ga_G(0)))+(n-2)\ep),
\]
which is less than $2n+2$ for $n\geq 2$, whenever $r_0$ and $\ep$ are such that $2(2r_0+(n-2)\ep)<1/2$.

\subsection{Proof of Lemma \ref{lemma:index+}}

First, we need the following result.

\begin{lemma}
If $\ep$ is rational then the set of periodic orbits of $\alpha$ is dense in $S^{2n+1}$.
\end{lemma}

\begin{proof}
Let, as in the previous section, $U_0=\{x\in U; -G(x)<r_0\}$. By the construction of $\alpha$, it is enough to show that the periodic orbits are dense in $\pi^{-1}(U_0)$. Consider the subset
\[
S=\{(q_1,p_1,\dots,q_n,p_n)\in U_0;\, F(q_2,p_2) \in \Q\}.
\]
This subset is dense $U_0$ and it is clear from \eqref{eq:flow G} that if $\ep$ is rational then every orbit of $G$ with initial condition in $S$ is closed and has rational period. Therefore, since the Hamiltonian flow of $H$ is a reparametrization of the Hamiltonian flow of $G$, every orbit of $H$ with initial condition in $S$ is closed.

Given $x \in S$, let $\ga_H$ be a closed orbit of $H$ such that $\ga_H(0)=x$. Let $\ga$ be an orbit of $\alpha$ whose projection is $\ga_H$. From \eqref{eq:action} we conclude that $\ga$ is a segment of a closed orbit if and only if $A_H(\ga_H)$ is rational (here we are using the fact that the action is homogeneous by iterations). Let $T$ be the period of $\ga_H$ and $T_G$ the period of the corresponding orbit of $G$. By \eqref{eq:action H},
\begin{align*}
A_H(\ga_H) & = T(f(G)-f'(G)(G-F^2)) \\
& = T_G\frac{f(G)-f'(G)(G-F^2)}{f'(G)} \\
& = T_G(s(G)+F^2),
\end{align*}
where, to simplify the notation, we are omitting the dependence of $G$ on $\ga_H(0)$ and of $F$ on $\tau_2(\ga_H(0))$ and, as in the previous section,
\[
s(r)=f(r)/f'(r)-r.
\]
Since $x \in S$, $T_G$ and $F^2$ are rational. Consequently, it is enough to show that the set
\[
C:=\{r \in (-r_0,0);\,s(r) \in \Q\}
\]
is dense in $(0,r_0)$. As a matter of fact, if $C$ is dense we can fix the second coordinate of $x$ (therefore remaining in $S$) and vary a little bit one of the other coordinates in order that $r$ belongs to $C$ and therefore the corresponding orbit has rational action.

But the density of $C$ follows from the fact that $s$ is smooth and $s'(r)< 0$ for every $r \in (-r_0,0)$.
\end{proof}

Thus, choosing $\ep$ rational we have that the set of periodic orbits of $\alpha$ is dense in $S^{2n+1}$. Consequently, by the continuity of the mean index with respect to the initial condition, it is enough to show that there exists a positive constant $b$ such that
\begin{equation}
\label{eq:index positive bga}
\mi(\ga)>bT
\end{equation}
for every periodic orbit $\ga$ with period $T$. It is clearly true if $\ga$ lies outside $\pi^{-1}(U_0)$. So suppose that $\ga$ lies in $\pi^{-1}(U_0)$ and let $\ga_H=\pi\circ\ga$ be the corresponding orbit of $H$. Clearly, the inequality \eqref{eq:index positive bga} holds if $\ga_H(t)\equiv 0$ so we can suppose that $\ga_H$ is non-constant. We have from \eqref{eq:index bga} that
\[
\mu(\ga) \geq q(2n+2) -\delta(T_G) - \delta(2F(\tau_2(\ga_G(0)))T_G) - (n-2)\delta(\ep T_G) -1,
\]
where $T_G$ is the period of the corresponding orbit of $G$ and $q$ is such that $[\ga]=q\mathfrak f$. Consequently,
\begin{align*}
\mi(\ga) & = \lim_{k\to\infty} \mu(\ga^k)/k \\
& \geq \lim_{k\to\infty} k^{-1}(kq(2n+2) -\delta(kT_G) - \delta(2F(\tau_2(\ga_G(0)))kT_G) - (n-2)\delta(\ep kT_G)).
\end{align*}
But it is easy to see that
\[
\lim_{k\to\infty} \delta(kx)/k = 2x
\]
for every $x\in \R$. Therefore,
\[
\mi(\ga) \geq q(2n+2) - 2T_G(1+2F(\tau_2(\ga_G(0))) + (n-2)\ep).
\]
We claim that $q>T/2$. In fact, by \eqref{eq:action H},
\[
q \geq T(f(G)-f'(G)G).
\]
However, the function $g(r):=f(r)-f'(r)r$ satisfies, by our choice of $f$, $g(-r_0)=C>1/2$ and
\[
g'(r) = -f''(r)r > 0
\]
for every $r\in (-r_0,0)$. Hence, $f(G)-f'(G)G > 1/2$ which implies that $q>T/2$.

Thus, since $T_G=f'(G(\ga_H(0)))T<T$,
\[
\mi(\ga) \geq T(n+1 - 2(1+2F(\tau_2(\ga_G(0))) + (n-2)\ep)).
\]
Take $r_0$ and $\ep$ sufficiently small such that $2(1+2r_0 + (n-2)\ep) < n+1$. Then,
\[
\mi(\ga)>bT
\]
with $b=n+1-2(1+2r_0+\ep(n-2))>0$.

\subsection{Proof of Lemma \ref{lemma:no new orbit}}
\label{sec:proof no new orbit}

Firstly, we need the following result on the Poincar\'e map $P: S \to P(S)$. Recall that $S=W\times \{0\}$ with the coordinates $(q_1,p_1,\dots,q_n,p_n,\theta)$ discussed before Lemma \ref{lemma:coords}. Consider a small disk $D \subset \R^2$ centered at the origin such that $D^n\times \{0\} \subset S$ and let $V=D^n \times S^1$. In what follows, we will identity $S$ with an open neighborhood of the origin in $\R^{2n}$. 

\begin{lemma}
\label{lemma:P}
There exist maps $Q_i: D^n \to D$ such that
\[
P(q_1,p_1,\dots,q_n,p_n) = (Q_1(q_1,p_1,\dots,q_n,p_n),\dots,Q_n(q_1,p_1,\dots,q_n,p_n))
\]
for every $(q_1,p_1,\dots,q_n,p_n) \in D^n$, where each $Q_i$ satisfies the following conditions:
\begin{itemize}
\item[i)] There are maps $\rho_i: D^n \to [0,2\pi)$ such that $Q_i(z_1,\dots,z_n)=e^{-\sqrt{-1}\rho_i(z_1,\dots,z_n)}z_i$, where we are identifying $(q_j,p_j)$ with $z_j=q_j+\sqrt{-1}p_j$ for every $j$.
\item[ii)] Taking $r_0<1/2$, $\ep<1$ and $D$ sufficiently small we have that $\rho_i(z_1,\dots,z_n)\neq 0$ whenever $z_i\neq 0$.
\item[iii)] Given $\phi>0$ we can choose $r_0$ and $\ep$ sufficiently small such that $\rho_i(z_1,\dots,z_n) \in [0,\phi)$ for every $2\leq i \leq n$. Moreover, we can choose $D$ sufficiently small such that if $z_2=0$ and $z_1\neq 0$ then $\rho_1(z_1,\dots,z_n) \in (\pi-\phi,\pi)$.
\end{itemize}
In particular, choosing $r_0$, $\ep$ and $D$ properly, we have that every periodic orbit of $P|_{\Sigma \cap V}$ different from the origin has period bigger than two.
\end{lemma}

\begin{remark}
The last assertion of the lemma also follows from Lemma \ref{lemma:action}. In fact, suppose that $P$ has a periodic orbit $x$ different from the origin with period two and let $\bga$ be the corresponding orbit of $\balpha$. Let $\ga$ be the lift of $\bga$ to $S^{2n+1}$ (which is a closed orbit of $\alpha$). Since $x$ has period two, we conclude by \eqref{eq:action} that $\ga$ has action equal to one, contradicting Lemma~\ref{lemma:action}.
\end{remark}

\begin{proof}
We have that
\[
P(x) = \bar\pi\circ\vr^{\balpha}_{T(x)}(x,0) = \vr_H^{T(x)}(x),
\]
where $\vr^{\balpha}_t$ is the Reeb flow of $\balpha$, $\bar\pi: V = D^n\times S^1 \to D^n$ is the projection and $T(x)$ is the return time to $\Sigma$. The existence of the maps $Q_i$ and item (i) are clear from  the construction of $\balpha$ and \eqref{eq:flow G}. In order to prove itens (ii) and (iii), note that, by \eqref{eq:action}, $T(x)$ is uniquely characterized by the equation
\begin{align*}
1/2 & = \int_ 0^{T(x)} d\theta(\bar\eta(t))\,dt \\
& = \int_0^{T(x)}  \beta(R_\balpha(\bar\eta(t))) - \lambda(X_{H}(\eta_H(t)))\,dt \\
&  = \int_0^{T(x)} H(\eta_H(t)) - \lambda(X_{H}(\eta_H(t)))\,dt \\
& =  \int_0^{T(x)} f(G(x)) - f'(G(x))\lambda(X_{G}(\eta_H(t)))\,dt \\
& = T(x)f(G(x)) - f'(G(x))\int_0^{T(x)} \lambda(X_{G}(\eta_H(t)))\,dt,
\end{align*}
where $\eta_H$ is the orbit of $H$ given by the projection of the Reeb orbit $\bar\eta$ of $\balpha$ satisfying $\bar\eta(0)=(x,0)$. As explained in the proof of Lemma \ref{lemma:action}, $\lambda(X_{G}(\eta_H(t)))$ does not depend on $t$ and therefore
\[
T(x)(f(G(x)) - f'(G(x))A) = 1/2,
\]
where $A=\lambda(X_{G}(\eta_H(0)))$. But, arguing as in the proof of Lemma \ref{lemma:index+}, we conclude that
\[
1/2 = T(f(G)-f'(G)A) \geq T(f(G)-f'(G)G) \geq T/2.
\]
Therefore, $T\leq 1$, where, for short, we are omitting the dependence of $G$ and $T$ on $x$. Assume that $r_0<1/2$ and $\ep<1$. Then it follows from \eqref{eq:flow G} and the properties of $f$ that every non-constant periodic orbit of $H$ has period bigger than one. Consequently, taking $D$ sufficiently small, we have that $\rho_i(z_1,\dots,z_n)\neq 0$ whenever $z_i\neq 0$. Moreover, the fact that $T\leq 1$ and \eqref{eq:flow G} show that given $\phi>0$ one can choose $r_0$ and $\ep$ sufficiently small such that $\rho_i(z_1,\dots,z_n) \in [0,\phi)$ for every $2\leq i \leq n$.

Now, suppose that $z_2=0$ and $z_1\neq 0$. Let $T_G=f'(G)T$ be the period of the corresponding orbit $\eta_G$ of $G$. By \eqref{eq:action H},
\[
1/2 = s(G)T_G,
\]
where $s(r)=f(r)/f'(r)-r$ satisfies $s(r)>1$ for every $r \in (-r_0,0)$ (note that $A=G$ because $z_2=0$). Thus, given $\delta>0$ one can choose $D$ sufficiently small such that $s(G)$ is close enough to one so that $T_G\in (1/2-\delta,1/2)$. By \eqref{eq:flow G}, this implies, choosing $D$ and $\delta$ sufficiently small, that $\rho_1(z_1,\dots,z_n) \in (\pi-\phi,\pi)$.
\end{proof}

Let $P_{\ep'}: S'' \to S'$ be the first return map of $\balpha_{\ep'}$, where $S'' \subset S'$ is chosen sufficiently small such that $P_{\ep'}$ is a well defined symplectic embedding. Lemma \ref{lemma:no new orbit} is an immediate consequence of the following result.

\begin{lemma}
There exist $\ep_0'>0$ and a neighborhood of the origin such that, for all $\ep'<\ep'_0$, the only 2-periodic orbit of $P_{\ep'}$ entirely contained in this neighborhood is the origin. 
\end{lemma}

\begin{proof}
Throughout the proof, we will use the coordinates established in Lemma \ref{lemma:coords}. For short, we will omit the superscript in $(q_1',p_1',\dots,q_n',p_n')$. First of all, note that $P'$ and $P_{\ep'}$ are given by the time $1/2$ maps of the Hamiltonian flows of $H_t$ and $H^{\ep'}_t$ respectively. Moreover, as explained in Section \ref{sec:proof main result}, by Lemma \ref{lemma:returns} we can assume that $P'=P$.

Arguing by contradiction, suppose that the lemma is not true. Then there exists a sequence of 2-periodic orbits $\{x_j=P^2_{\ep'}(x_j),P_{\ep'}(x_j)\}$ converging to the origin such that $x_j \neq 0$. Take $U \subset B'$ sufficiently small such that
\[
P^2_{\ep'}(x)=dP^2_{\ep'}(0)x+{\mathcal R}_{\ep'}(x),
\]
for every $x \in U$, where ${\mathcal R}_{\ep'}$ satisfies
\[
\lim_{x\to 0} \frac{{\mathcal R}_{\ep'}(x)}{\|x\|}=0.
\]
Extracting a subsequence if necessary, suppose that $x_j/\|x_j\| \to v$. Then
\[
v = \lim_{j\to\infty} \frac{x_j}{\|x_j\|} = \lim_{j\to\infty} \frac{P^2_{\ep'}(x_j)}{\|x_j\|} = dP^2_{\ep'}(0)v,
\]
which implies that $v \in \ker(dP^2_{\ep'}(0)-\Id)=\text{span}\{\partial q_1,\partial q_2\}$. Thus, given $\theta>0$ there exists $j_0$ such that $x_j$ lies in the subset
\[
C_\theta=\{(q_1,p_1,\dots,q_n,p_n); \frac{\la (q_1,p_1,\dots,q_n,p_n),v \ra}{\|(q_1,p_1,\dots,q_n,p_n)\|\|v\|} \in (1-\theta,1+\theta)\ \text{for some}\ v \in \text{span}\{\partial q_1,\partial q_2\}\}
\]
for every $j>j_0$. From now on, to simplify the notation, we will omit the subscript $j$.

Define the Hamiltonian
\[
\S_t(q_1,p_1,\dots,q_n,p_n) = \ep'\chi'(t)S(q_1,p_1,\dots,q_n,p_n)
\]
so that
\[
H^{\ep'}_t|_{B'} = (H_t \# \S_t)|_{B'}.
\]
Let $W$ be a sufficiently small neighborhood of the origin such that $W \subset B'$, $W \subset D^n$ (where $D$ is given by the previous lemma) and every periodic orbit of $H^{\ep'}_t$ with period two and initial condition in $W$ is contained in $B'$. Therefore,
\[
P_{\ep'}^2(x)=(P\circ \vr^{\S_t}_{1/2})^2(x).
\]
Write
\[
x=(q_1^0,p_1^0,\dots,q_n^0,p_n^0),\, P_{\ep'}(x)=(q_1^{1/2},p_1^{1/2},\dots,q_n^{1/2},p_n^{1/2})\ \text{and}\ P_{\ep'}^2(x)=(q_1^1,p_1^1,\dots,q_n^1,p_n^1).
\]
It follows from Lemma \ref{lemma:P} and the identity
\[
\vr^{\S_t}_{1/2}(q_1,p_1,\dots,q_n,p_n)=(q_1+\ep'p_1/2,p_1,q_2+\ep'p_2/2,p_2,q_3,p_3,\dots,q_n,p_n)
\]
that $(q_i^0,p_i^0)=(0,0)$ for every $i \in \{3,\dots,n\}$ (because, for all $i \in \{3,\dots,n\}$, $Q_i$ is a small rotation and $\tau_i\circ\vr_{1/2}^{\S}(q_1,p_1,\dots,q_n,p_n)=(q_i,p_i)$, where $\tau_i (q_1,p_1,\dots,q_n,p_n)=(q_i,p_i)$ denotes the projection onto the $i$-th factor). Hence, $(q_i^0,p_i^0)\neq (0,0)$ for some $i \in \{1,2\}$.

We claim that $(q_2^0,p_2^0)=(0,0)$. As a matter of fact, if $(q_2^0,p_2^0)\neq (0,0)$ then, by Lemma \ref{lemma:P}, choosing $r_0$ sufficiently small we have that $\tau_2\circ P(x)$ is a small rotation of $(q_2^0,p_2^0)$. Since $x$ lies in $C_\theta$, one can check that, choosing $\theta$ sufficiently small, $(P\circ \vr^{\S_t}_{1/2})^2(x)\neq x$, a contradiction.

Thus, $x$ lies in the plane $q_2=p_2=\dots=q_n=p_n=0$. Choose $\theta$, $\ep_0'$, $r_0$ and $W$ sufficiently small such that
\[
P(C_\theta \cap W) \subset C_{2\theta} \quad\text{and}\quad \vr^{\S_t}_{1/2}(C_\theta \cap W) \subset C_{2\theta}
\]
for every $\ep'<\ep_0'$. The existence of $\theta$, $\ep_0'$, $r_0$ and $W$ readily follows from Lemma \ref{lemma:P} and the explicit description of $\vr^{\S_t}_{1/2}$. 

We shall show that $(q_1^0,p_1^0)=(0,0)$, contradicting our assumption on $x$. In order to prove this, define
\[
C^+_\theta = \{(q_1,p_1,\dots,q_n,p_n) \in C_\theta;\, q_1p_1 > 0\}\quad\text{and}\quad C^-_\theta = \{(q_1,p_1,\dots,q_n,p_n) \in C_\theta;\, q_1p_1 < 0\}.
\]
Since
\[
\tau_1\circ \vr^{\S_t}_{1/2}(q_1,p_1,\dots,q_n,p_n)=(q_1+\ep'p_1/2,p_1),
\]
we have that, choosing $\theta$ sufficiently small,
\begin{equation}
\label{eq:F1}
F(\tau_1\circ \vr^{\S_t}_{1/2}(y)) < F(y)
\end{equation}
for every $y \in C^-_{2\theta} \cap W \cap \{q_2=\dots=p_n=0\}$ and
\begin{equation}
\label{eq:F2}
F(\tau_1\circ \vr^{\S_t}_{1/2}(y)) > F(y)
\end{equation}
for every $y \in C^+_{2\theta} \cap W \cap \{q_2=\dots=p_n=0\}$.

We claim that if $(q_1^0,p_1^0)\neq (0,0)$ then $P_{\ep'}^2(x)\neq x$. Indeed, notice that 
\begin{equation}
\label{eq:F3}
F(\tau_1\circ P(y))=F(\tau_1(y))
\end{equation}
for every $y$ liying in the plane $q_2=\dots=p_n=0$. Let us consider the following three possible cases:

\vskip .3cm
\noindent {\bf (I) $(q_1^0,p_1^0,0,\dots,0) \in C_\theta^-$.}
\vskip .2cm

In this case, by \eqref{eq:F1} and \eqref{eq:F3},
\[
F(\tau_1(P_{\ep'}(x)) < F(\tau_1(x)).
\]
If $p_1^{1/2}=0$ then
\[
F(\tau_1(P_{\ep'}^2(x)))=F(\tau_1(P_{\ep'}(x))) < F(\tau_1(x))
\]
implying that $P_{\ep'}^2(x)\neq x$. If $\sign p_1^{1/2}=\sign p_1^{0}$ then $P_{\ep'}^2(x)\neq x$ because
\[
P_{\ep'}^2(x)=P\circ\vr^{\S_t}_{1/2}(P_{\ep'}(x))
\]
and $P$ rotates $\vr^{\S_t}_{1/2}(P_{\ep'}(x))$ clockwise with an angle less than $\pi$ (in the plane $q_2=\dots=p_n=0$). If $\sign p_1^{1/2}=-\sign p_1^{0}$ then, by \eqref{eq:F1} and \eqref{eq:F3},
\[
F(\tau_1(P_{\ep'}^2(x))) = F(\tau_1\circ \vr^{\S_t}_{1/2}(P_{\ep'}(x)))) < F(\tau_1(P_{\ep'}(x))) < F(\tau_1(x)),
\]
and consequently $P_{\ep'}^2(x)\neq x$. (Note that, since $P$ rotates $x$ with an angle less than and close to $\pi$ and $x$ lies in $C_\theta$, $\sign q_1^{1/2}=-\sign q_1^{0}$.) 

\vskip .3cm
\noindent {\bf (II) $(q_1^0,p_1^0,0,\dots,0) \in C_\theta^+$.}
\vskip .2cm

By \eqref{eq:F2} and \eqref{eq:F3},
\[
F(\tau_1(P_{\ep'}(x))) > F(\tau_1(x)).
\]
Since $P$ rotates $\vr^{\S_t}_{1/2}(x)$ clockwise with an angle less than and close to $\pi$, $\sign p_1^{1/2}=-\sign p_1^{0}$. Hence, by \eqref{eq:F1} and \eqref{eq:F3},
\[
F(\tau_1(P_{\ep'}^2(x))) = F(\tau_1\circ \vr_{1/2}^{\S_t}(P_{\ep'}(x))) > F(\tau_1(P_{\ep'}(x))) > F(\tau_1(x))
\]
and, again, $P_{\ep'}^2(x)\neq x$. (Once more, we are using that $\sign q_1^{1/2}=-\sign q_1^{0}$.)

\vskip .3cm
\noindent {\bf (III) $p_1^0=0$.}
\vskip .2cm

In this last case, we have that
\[
(q_1^{1/2},p_1^{1/2})=\tau_1(P(x))=\tau_1(P_{\ep'}(x))
\]
satisfies $p_1^{1/2}\neq 0$ (since $P$ rotates $x$ with an angle less than and close to $\pi$) and therefore $F(\tau_1\circ \vr_{1/2}^{\S_t}(P_{\ep'}(x))) \neq F(\tau_1\circ P_{\ep'}(x))$. But, by \eqref{eq:F3},
\[
F(\tau_1(P_{\ep'}^2(x))) = F(\tau_1\circ \vr_{1/2}^{\S_t}(P_{\ep'}(x))) \neq F(\tau_1(P_{\ep'}(x))) = F(\tau_1(P(x))) = F(\tau_1(x))
\]
implying that $P_{\ep'}^2(x)\neq x$ and finishing the proof of the lemma.

\end{proof}

\subsection{Proof of Lemma \ref{lemma:gamma0}}

Let $R: [0,1/2] \to \Sp(2)$, $R_\ep: [0,1/2] \to \Sp(2)$ and $S_{\ep'}: [0,1/2] \to \Sp(2)$ be the paths given by $R(t)=e^{2\pi\sqrt{-1}t}$, $R_\ep(t)=e^{2\pi\ep\sqrt{-1}t}$ and
\[
S(t) = \left(
\begin{matrix}
1 & \ep' t \\
0 & 1
\end{matrix} \right),
\]
where $\ep$ is given by the construction of $\alpha$. Let $\Ga$ be the linearized Reeb flow of $\alpha_{\ep'}$ along $\ga_0^{\ep'}$ with respect to the trivialization $\Phi$ defined in Section \ref{sec:dynconvex}. It is clear from the construction of $H^{\ep'}_t$ and the fact that it is $1/2$-periodic that, with $P(t):=R(t)\circ S_{\ep'}(t)$, the linearized Hamiltonian flow of $H^{\ep'}_t$ over the constant solution $\ga_0(t)\equiv 0$ is given by
\[
\Ga(t)= P(t) \oplus S_{\ep'}(t) \oplus \underbrace{R_\ep(t) \oplus \dots \oplus R_\ep(t)}_{n-2\ \text{times}}
\]
for $t \in [0,1/2]$ and $\Ga(t+1/2)=\Ga(t)\circ\Ga(1/2)$ for all $t$. We claim that 
$$
\mu(P)=-1, \, \mu(P^2)=-2, \, b_-(P^2(1))=0 \textrm{ and }b_+(P^2(1))=1.
$$
Indeed, the first equality follows from the fact that $P$ is a small perturbation of $R$, $\mu(R)=-1$ and $R$ is non-degenerate. To compute $\mu(P^2)$, let $\Bott_P: S^1 \to \Z$ be Bott's function associated to $P$. We have that $\Bott_P(z)=-1$ for every $z\neq -1$ because $\Bott_P(1)=\mu(P)=-1$ and $\Bott_P$ is constant on $S^1\setminus \{-1\}$. A computation shows that the splitting numbers at $-1$ are given by
\[
S^\pm_{-1}(P(1/2))=S^\pm_{1}(P^2(1))=0.
\]
Therefore, $\Bott(-1)=-1$ which implies that $\mu(P^2)=-2$. Finally, by the definition of $b_\pm$, it is clear that $b_-(P^2(1))=0$ and $b_+(P^2(1))=1$.

Now, we claim that $\mu(S_{\ep'})=\mu(S_{\ep'}^2)=0$ and $b_-(S_{\ep'}^2(1))=0$ and $b_+(S_{\ep'}^2(1))=1$. In fact, let $\Bott_S: S^1 \to \Z$ be Bott's function associated to $S_{\ep'}$. The spectrum of $S_{\ep'}(t)$ is constant equal to one for every $t$, implying that $\Bott_S$ is constant on $S^1\setminus \{1\}$ and that the mean index of $S_{\ep'}$ given by
\[
\int_{S^1} \Bott_S(z)\,dz
\]
vanishes. As before, a computation shows that
\[
S^\pm_{1}(S_{\ep'}(1/2))=0.
\]
Hence, $\Bott_S(z)=0$ for every $z$ and consequently $\mu(S_{\ep'})=\mu(S_{\ep'}^2)=0$ (actually, $\mu(S_{\ep'}^k)=0$ for every $k$). Finally, by the definition of $b_\pm$, it is clear that $b_-(S_{\ep'}^2(1))=0$ and $b_+(S_{\ep'}^2(1))=1$.

Thus, since $\ep$ is very small (in particular, less than one),
\begin{align*}
\mu(\Ga^2) & =\mu(P^2)+\mu(S_{\ep'}^2)+(n-2)\mu(R_\ep^2) \\
& =-2-n+2 \\
& =-n
\end{align*}
implying, by Lemma \ref{lemma:triv}, that
\[
\mu(\ga_0^{\ep'})=-n+2n+2=n+2.
\]
Finally, by the discussion above,
\[
\mu(\ga^{\ep'}_0(1))+b_-(\ga^{\ep'}_0)-b_+(\ga^{\ep'}_0(1))=n+2-2=n.
\]

\section{Proof of Theorem \ref{thm:example open}}
\label{sec:example open}

\subsection{Generalization of the theorem}
For $n\geq 2$, the proof of Theorem \ref{thm:example} gives us a
dynamically convex and antipodally symmetric contact form $\alpha$ on
$S^{2n+1}$ with a symmetric periodic orbit $\ga$ of period one which is not strongly dynamically convex, i.e.,
\[
\cz(\ga) + b_-(\ga(1)) - b_+(\ga(1)) < n+2.
\]
Consequently, by Theorem \ref{thm:c=>sdc1}, $\psi^*\alpha$ cannot be convex for any contactomorphism $\psi: S^{2n+1}\to S^{2n+1}$ that commutes with the antipodal map.

When $n$ is odd, we have the following, somewhat technical, generalization of Theorem \ref{thm:c=>sdc1}.

\begin{theorem}
\label{thm:c=>sdc2}
Let $\phi_t: \R^{2n+2}\setminus\{0\}\to \R^{2n+2}\setminus\{0\}$, with $t \in \R/\Z$, be a Hamiltonian circle action whose orbits have zero Maslov index. Denote by $\phi=\phi_{1/2}$ the generator of the induced $\Z_2$-action. Let $H: \R^{2n+2}\setminus\{0\} \to \R$ be a convex homogeneous of degree two Hamiltonian $\phi$-invariant  and assume that $\Sigma:=H^{-1}(1)$ is a regular energy level. Let $\ga$ be  a symmetric closed orbit of $H$ on $\Sigma$ of period $T$. Suppose that there exists $x \in \ga(\R)$ such that the linearized Hamiltonian circle action $t \mapsto \Phi(t):=d\phi_t(x)$ along the half orbit $\phi|_{[0,1/2]}(x)$ satisfies
\[
\Bott_\Phi(-1) \leq -(n+1),
\]
where $\Bott_\Phi$ is the Bott's function associated to $\Phi$. Then $\cz(\ga) + b_-(\ga(T)) - b_+(\ga(T)) \geq n+1$. Here, $\cz(\ga)$ denotes the index of $\ga$ viewed as a Hamiltonian closed orbit and we are using the usual (constant) symplectic trivialization of $T\R^{2n+2}$.
\end{theorem}

Postponing its proof to Section \ref{sec:c=>sdc2}, let us explain why the previous theorem generalizes Theorem \ref{thm:c=>sdc1} when $n$ is odd. Under this assumption, it turns out that the $\Z_2$-action on $\R^{2n+2}$ generated by the antipodal map can be induced by a Hamiltonian $S^1$-action $\phi_t: \R^{2n+2} \to \R^{2n+2}$ with zero Maslov index. (This is related to the fact that the induced contact structure on $\RP^{2n+1}$ has vanishing first Chern class if and only if $n$ is odd.) As a matter of fact, let $\theta_i=1$ for every $i \in \{1,\dots,(n+1)/2\}$ and $\theta_i=-1$ for every $i \in \{(n+1)/2+1,\dots,(n+1)\}$ (note that, by our assumptions, $(n+1)/2$ is a positive integer). Consider the Hamiltonian $S^1$-action on $\R^{2n+2} \cong \C^{n+1}$ given by
\begin{equation}
\label{eq:Phi linear}
\phi_t(z_1,\dots,z_{n+1})=(e^{2\pi\sqrt{-1}\theta_1 t}z_1,\dots,e^{2\pi\sqrt{-1}\theta_{n+1} t}z_{n+1}),
\end{equation}
where $t \in S^1=\R/\Z$. Clearly, $\phi_{1/2}$ is the antipodal map. Moreover, since $\sum_i \theta_i = 0$, the Maslov index of every orbit of $\phi_t$ vanishes.

This circle action has the following property: given \emph{any} $x \in \R^{2n+2}$ the linearized Hamiltonian flow $t \mapsto \Phi(t):=d\phi_t(x)=\phi_t$ along the half orbit $\phi|_{[0,1/2]}(x)$ satisfies
\begin{equation}
\label{eq:Bott phi at -1}
\Bott_\Phi(-1) = -(n+1).
\end{equation}
Indeed, a straightforward computation shows that $\cz(\Phi)=0$ and $\cz(\Phi^2)=-(n+1)$. Hence, we conclude Theorem \ref{thm:c=>sdc1} from Theorem \ref{thm:c=>sdc2} and Proposition \ref{prop:ReebHam} when $n$ is odd.

In order to show that, when $n$ is odd, the example furnished by Theorem \ref{thm:example} has the properties stated in Theorem \ref{thm:example open} we proceed as follows. Let $\ga$ be the aforementioned symmetric orbit of $\alpha$ such that $\cz(\ga) + b_-(\ga(1)) - b_+(\ga(1)) < n+2$ and fix $x=\ga(0)$. We need the following lemma whose proof is given in Section \ref{sec:proof conjugation}. Note that every contactomorphism of $S^{2n+1}$ lifts to a symplectomorphism of its symplectization $SS^{2n+1} \simeq \R^{2n+2}\setminus\{0\}$.

\begin{lemma}
\label{lemma:conjugation}
There exists a contactomorphism $\vr: S^{2n+1} \to S^{2n+1}$ arbitrarily $C^1$-close to the identity such that the corresponding lifted symplectomorphism $\psi: \R^{2n+2}\setminus\{0\}\to \R^{2n+2}\setminus\{0\}$ has the following property. The conjugated action $\phi'_t:=\psi^{-1}\phi_t\psi$ satisfies
\[
\Bott^+_{\Phi'}(-1) \leq -(n+1),
\]
where $\Phi'(t):=d\phi'_t(\psi^{-1}(x))$ is the linearized circle action along the half orbit $\phi'|_{[0,1/2]}(\psi^{-1}(x))$ and $\Bott^+$ is the upper semicontinuous Bott's function defined in \eqref{eq:Bott+}.
\end{lemma}

It follows from the upper semicontinuity of the map $\Ga \mapsto \Bott^+_\Ga(-1)$ in the $C^0$-topology (see Section \ref{sec:Bott}) that the properties of the previous lemma hold for every contactomorphism $\vr': S^{2n+1} \to S^{2n+1}$ $C^1$-close to $\vr$. Therefore, we get a $C^1$-open subset $V \subset \Cont(S^{2n+1})$ such that every $\vr' \in V$ satisfies the properties of the lemma. Moreover, the closure of $V$ in the $C^1$-topology contains the identity.

By the lower semicontinuity of $\Bott_\Ga(-1)$ with respect to $\Ga$ in the $C^0$-topology, we conclude that
\[
\Bott_{\Phi'}(-1) \leq -(n+1)
\]
for any $\vr' \in \widebar V$. Thus, by Theorem \ref{thm:c=>sdc2}, $(\vr')^*\alpha$ cannot be convex for any $\vr' \in \widebar V$.

Now, let $\bvr \in S$ and $\balpha=\bvr^*\alpha$. The form
$\balpha$ is antipodally symmetric and has a symmetric closed orbit $\bga = \bvr^{-1}\ga$ such that 
$$
\cz(\bga) + b_-(\bga(1)) - b_+(\bga(1)) < n+2.
$$
Arguing as above, we get a $C^1$-open subset $V_\bvr \subset \Cont(S^{2n+1})$ such that the closure of $V_\bvr$ (in the $C^1$-topology) contains the identity and $(\vr')^*\balpha$ cannot be convex for any $\vr'$ in $V_\bvr$. Note that $V_{\id}=V$.

In this way, we obtain a $C^1$-open subset $U \subset \Cont(S^{2n+1})$ given by
\[
U = \bigcup_{\bvr \in S}\bigcup_{\vr' \in V_\bvr} \{\bvr\vr'\}
\]
such that $(\bvr')^*\alpha$ cannot be convex for any $\bvr' \in U$. Clearly, $S$ belongs to the $C^1$-closure of $V$. Moreover, again by the lower semicontinuity of $\Bott_\Ga(-1)$ with respect to $\Ga$ in the $C^0$-topology, we conclude that $(\bvr')^*\alpha$ cannot be convex for any $\bvr' \in \widebar U$.

\subsection{Proof of Theorem \ref{thm:c=>sdc2}}
\label{sec:c=>sdc2}

\subsubsection{Idea of the proof}
Let $\Ga$ be the linearized Hamiltonian flow of $H$ along $\ga$. In the proof of Theorem \ref{thm:c=>sdc1}, we used in a crucial way the fact that if $H$ is invariant under the antipodal map then $\Ga$ is the second iterate of $\Ga|_{[0,T/2]}$. Unfortunately, it is not true in general when $\phi$ is not the antipodal map. However, under our hypotheses, we can show that $\Ga$ is \emph{homotopic} to a path $\Psi$ such that $\Psi$ is the second iterate of $\Psi|_{[0,T/2]}$ (see Proposition \ref{prop:closing}). The path $\bPsi:=\Psi|_{[0,T/2]}$ is no longer positive in general, but our assumption that $\Bott_\Phi(-1) \leq -(n+1)$ and the convexity of $H$ imply that its index has the following crucial jump:
\[
\cz(\bPsi^2) - \cz(\bPsi) \geq n+1,
\]
see Proposition \ref{prop:comparison}. This condition allows us to show that $\cz(\bPsi^2) + b_-(\bPsi^2(T)) - b_+(\bPsi^2(T)) \geq n+1$ (see Proposition \ref{prop:sdc_jump}). Since $\Psi=\bPsi^2$ is homotopic (with fixed endpoints) to $\Ga$, we conclude Theorem \ref{thm:c=>sdc2}.

\subsubsection{Proof of the theorem}
As discussed above, let $\Ga$ be the linearized Hamiltonian flow of $H$ along $\ga$. Assume, without loss of generality, that $T=1$ and $x=\ga(0)$. Firstly, we need the following proposition. It uses only the fact that the orbits of $\phi$ have zero Maslov index.

\begin{proposition}
\label{prop:closing}
We have that $\Ga$ is homotopic with fixed endpoints to a path $\Psi: [0,1] \to \Sp(2n+2)$ such that $\Psi=\bPsi^2$, where $\bPsi:=\Psi|_{[0,1/2]}$.
\end{proposition}

\begin{proof}
Let $\vr^H_t$ be the Hamiltonian flow of $H$. Define $\tGa: \R^2 \to \Sp(2n+2)$, $\tPhi: \R^2 \to \Sp(2n+2)$ and $\Psi: \R \to \Sp(2n+2)$ as
\begin{equation}
\label{eq:tGa}
\tGa(t,s)=d\vr^H_t(\ga(s)),
\end{equation}
\begin{equation}
\label{eq:tPhi}
\tPhi(t,s)=d\phi_t(\ga(s))
\end{equation}
and
\begin{equation}
\label{eq:Psi}
\Psi(t)=\tPhi(t,0)^{-1} \circ \tGa(t,0).
\end{equation}
We have that $\Ga(t)=\tGa(t,0)$ and the symplectic loop $t \mapsto \tPhi(t,0)^{-1}$ has zero Maslov index. Therefore, $\Psi$ is homotopic to $\Ga$.

Using the relations $\vr^H_{t+t'}=\vr^H_t \circ \vr^H_{t'}$ and $\rho_{t+t'}=\rho_t \circ \rho_{t'}$ we arrive at
\begin{equation}
\label{eq:condition tGa and tPhi}
\tGa(t+t',s) = \tGa(t,s+t')\tGa(t',s)\quad\text{and}\quad  \tPhi(t+t',s) = \tPhi(t,s+t')\tPhi(t',s)
\end{equation}
for every $t$ and $t'$. Differentiating the relation $\vr^H_t \circ \phi_{1/2}(\ga(0)) = \phi_{1/2} \circ \vr^H_t(\ga(0))$ and using the fact that $\phi_{1/2}(\ga(0))=\ga(1/2)$ we conclude that $\tGa(t,1/2)\tPhi(1/2,0) = \tPhi(1/2,t)\tGa(t,0)$ and consequently
\begin{equation}
\label{eq:symmetry}
\tGa(t,0)\tPhi(1/2,0)^{-1} = \tPhi(1/2,t)^{-1}\tGa(t,1/2)
\end{equation}
for every $t \in \R$. Thus, we have that
\begin{align*}
\Psi(t)\Psi(1/2) & = \tPhi(t,0)^{-1}\tGa(t,0)\tPhi(1/2,0)^{-1}\tGa(1/2,0) \\
& = \tPhi(t,0)^{-1}\tPhi(1/2,t)^{-1}\tGa(t,1/2)\tGa(1/2,0) \\
& = (\tPhi(1/2,t)\tPhi(t,0))^{-1}\tGa(t,1/2)\tGa(1/2,0) \\
& = \tPhi(t+1/2,0)^{-1}\tGa(t+1/2,0) \\
& = \Psi(t+1/2)
\end{align*}
for every $t$, where the second and fourth equalities follow from \eqref{eq:symmetry} and \eqref{eq:condition tGa and tPhi} respectively. Hence, $\Psi$ is our desired path and we conclude the result.
\end{proof}

The path $\Psi$ is not positive in general. However, its index satisfies the following crucial jump condition.

\begin{proposition}
\label{prop:comparison}
The path $\Psi$ established in Proposition \ref{prop:closing} can be chosen such that
\[
\cz(\bPsi^2) - \cz(\bPsi) \geq n+1,
\]
where $\bPsi:=\Psi|_{[0,1/2]}$.
\end{proposition}

\begin{proof}
As explained in the proof of Proposition \ref{prop:closing},
\[
\Psi(t)=\Phi(t)^{-1} \circ \Ga(t)
\]
with $\Phi(t):=\tPhi(t,0)$ given by \eqref{eq:tPhi}. Let $A$, $B$ and $C$ be the path of symmetric matrices uniquely defined by the equations
\begin{equation}
\label{eq:deGa}
\frac{d}{dt}\Ga(t) = JA(t)\Ga(t),
\end{equation}
\begin{equation}
\label{eq:detPhi}
\frac{d}{dt}(\Phi(t)^{-1}) = JB(t)\Phi(t)^{-1}
\end{equation}
and
\begin{equation}
\label{eq:dePsi}
\frac{d}{dt}\Psi(t)= JC(t)\Psi(t),
\end{equation}
respectively.
\begin{lemma}
Let $\Ga$ and $\Phi$ be symplectic paths starting at the identity and satisfying equations \eqref{eq:deGa} and \eqref{eq:detPhi} respectively. Let $\Psi$ be the symplectic path given by $\Psi(t)=\Phi(t)^{-1} \circ \Ga(t)$ and satisfying \eqref{eq:dePsi}. We have that $C(t)=B(t)+\Phi(t)^*A(t)\Phi(t)$ for all $t$, where $\Phi(t)^*$ denotes the transpose of $\Phi(t)$.
\end{lemma}

\begin{proof}
It follows from the equation
\begin{align*}
\frac{d}{dt}\Psi(t) & = \frac{d}{dt}(\Phi(t)^{-1})\Ga(t) + \Phi(t)^{-1}\frac{d}{dt}\Ga(t) \\
& = JB(t)\Phi(t)^{-1}\Ga(t) + \Phi(t)^{-1}JA(t)\Ga(t) \\
& = J(B(t)+\Phi(t)^*A(t)\Phi(t))\Psi(t),
\end{align*}
where the last equality follows from the definition of $\Psi$ and the fact that
\begin{align*}
\Phi(t)^{-1}JA(t)\Ga(t) & = J(-J\Phi(t)^{-1}JA(t)\Phi(t))\Psi(t) \\
& = J(\Phi(t)^*A(t)\Phi(t))\Psi(t),
\end{align*}
where we used that $\Phi(t)$ is symplectic and therefore $\Phi(t)^*=-J\Phi(t)^{-1}J$.
\end{proof}
By our assumptions, $\Bott_{\bPhi}(-1) \leq - (n+1)$, where $\bPhi:=\Phi|_{[0,1/2]}$. We have, from \eqref{eq:Bott x Bott+}, that $\Bott_\bPhi(-1)=-\Bott^+_{\bPhi^{-1}}(-1)$. Therefore,
\begin{equation}
\label{eq:hypothesis Bott}
\Bott^+_{\bPhi^{-1}}(-1) \geq n+1.
\end{equation}
This implies that, given $\ep>0$, there exists a symplectic path $\bPsi_\ep$ starting at the identity $C^1$ $\ep$-close to $\bPhi^{-1}$ such that
\begin{equation}
\label{eq:Bott Psi_ep}
\Bott_{\bPsi_\ep}(-1) \geq n+1.
\end{equation}
Let $\Psi_\ep=\bPsi_\ep^2$ and $C_\ep$ be the path of symmetric matrices such that
\[
\frac{d}{dt} \Psi_\ep(t) = JC_\ep(t)\Psi_\ep(t).
\]
Define $\Theta_\ep(t)=\Phi(t) \circ \Psi_\ep(t)$ and let $D_\ep$ be the path of symmetric matrices such that
\[
\frac{d}{dt} \Theta_\ep(t) = JD_\ep(t)\Theta_\ep(t).
\]
Since $\Psi_\ep(t)=\Phi(t)^{-1}\circ\Theta_\ep(t)$, we see from the previous lemma that
\[
C_\ep(t) = B(t)+\Phi(t)^*D_\ep(t)\Phi(t).
\]
The fact that $A(t)=d^2H(\ga(t))$ and $H$ is convex on $\ga$, allows us to choose $\ep$ such that $C(t) - C_\ep(t) > 0$ for every $t$. Indeed, take $\ep$ small enough such that $A(t) - D_\ep(t) > 0$ for all $t$ ($D_\ep(t)$ is arbitrarily small as $\ep\to 0$ because $\Psi_\ep$ is $C^1$ $\ep$-close to $\Phi^{-1}$) and note that
\begin{align*}
\lg (C(t) - C_\ep(t))v,v \rg & = \lg \Phi(t)^*(A(t)-D_\ep(t))\Phi(t) v,v \rg \\
& = \lg (A(t)-D_\ep(t))\Phi(t) v,\Phi(t)v \rg > 0
\end{align*}
for every $t$ and $v \neq 0$.

Then, it follows from the previous discussion, \eqref{eq:Bott Psi_ep} and Theorem \ref{thm:comparison} that
\[
\cz(\bPsi^2) - \cz(\bPsi) \geq \cz(\bPsi_\ep^2) - \cz(\bPsi_\ep) = \Bott_{\bPsi_\ep}(-1) \geq n+1.
\]
\end{proof}

Hence, Theorem \ref{thm:c=>sdc2} follows from Lemma \ref{lemma:linear}, Propositions \ref{prop:ReebHam}, \ref{prop:closing}, \ref{prop:comparison} and the following result.

\begin{proposition}
\label{prop:sdc_jump}
Let $\bPsi: [0,1/2] \to \Sp(2n+2)$ be a symplectic path starting at the identity. Suppose that
\[
\cz(\bPsi^2) - \cz(\bPsi) \geq n+1.
\]
Then
\[
\cz(\bPsi^2) + 2S^+_1(\bPsi^2(1)) - \nu(\bPsi^2(1)) \geq n+1.
\]
\end{proposition}

\begin{proof}
Let $\Bott$ be the Bott's function associated to $\bPsi$ and $P=\bPsi(1/2)$. We have from Bott's formula, \eqref{eq:bott2}, \eqref{eq:bott3} and our assumptions that
\begin{align}
\label{eq:est1}
\begin{split}
\cz(\bPsi^2) + 2S^+_1(P^2) - \nu(P^2) & = \Bott(1) + \Bott(-1) + 2\sum_{z^2=1} S^+_z(P) - \sum_{z^2=1} \nu_z(P) \\
& \geq \Bott(1) + n+1 + 2\sum_{z^2=1} S^+_z(P) - \sum_{z^2=1} \nu_z(P).
\end{split}
\end{align}
In what follows, for the sake of simplicity, we will omit the dependence of $S^\pm_z$ and $\nu_z$ on $P$. From \eqref{eq:bott via splitting}, we get
\[
\Bott(-1) = \Bott(1) + S^+_1 + \sum_{\phi \in (0,\pi)}(S^+_{e^{\sqrt{-1}\phi}}-S^-_{e^{\sqrt{-1}\phi}}) - S^-_{-1},
\]
which implies that
\begin{align*}
\Bott(1) & = \Bott(-1) - S^+_1 - \sum_{\phi \in (0,\pi)}(S^+_{e^{\sqrt{-1}\phi}}-S^-_{e^{\sqrt{-1}\phi}}) + S^-_{-1}  \\
& \geq n+1 - S^+_1 - \sum_{\phi \in (0,\pi)}(S^+_{e^{\sqrt{-1}\phi}}-S^-_{e^{\sqrt{-1}\phi}}) + S^-_{-1}.
\end{align*}
Plugging this inequality in \eqref{eq:est1} we arrive at
\begin{align*}
\cz(\bPsi^2) + 2S^+_1(P^2) - \nu(P^2) & \geq  2n+2 - S^+_1 + 2\sum_{z^2=1} S^+_z - \sum_{z^2=1} \nu_z - \sum_{\phi \in (0,\pi)}(S^+_{e^{\sqrt{-1}\phi}}-S^-_{e^{\sqrt{-1}\phi}}) + S^-_{-1} \\
& \geq 2n+2 - \sum_{z^2=1} ( \nu_z - S^+_z) - \sum_{\phi \in (0,\pi)}(S^+_{e^{\sqrt{-1}\phi}}-S^-_{e^{\sqrt{-1}\phi}}),
\end{align*}
where the last inequality uses the fact that the splitting numbers are non-negative. Thus, it is enough to show that
\begin{equation}
\label{eq:desired bound}
\sum_{z^2=1} ( \nu_z - S^+_z) + \sum_{\phi \in (0,\pi)}(S^+_{e^{\sqrt{-1}\phi}}-S^-_{e^{\sqrt{-1}\phi}}) \leq n+1.
\end{equation}
In order to prove this inequality, define
\[
V = \bigoplus_{z \in \sigma(P) \cap S^1;\,z^2=1} E_z\quad\text{and}\quad W = \bigoplus_{z \in \sigma(P) \cap S^1;\,z^2\neq 1} E_z,
\]
where $\sigma(P)$ denotes the spectrum of $P$ and $E_z$ is the subspace whose complexification is the generalized eigenspace of $P$ associate to $z$. Clearly,
\begin{equation}
\label{eq:dimensions}
\dim V + \dim W \leq 2n+2.
\end{equation}
Note that, by \eqref{eq:bound nullity-splitting},
\begin{equation}
\label{eq:1st term}
\sum_{z^2=1} \nu_z - S^+_z \leq \frac{\dim V}{2}.
\end{equation}
On the other hand, by \eqref{eq:bound splitting},
\begin{equation}
\label{eq:2nd term}
\sum_{\phi \in (0,\pi)}(S^+_{e^{\sqrt{-1}\phi}}-S^-_{e^{\sqrt{-1}\phi}}) \leq \frac{\dim W}{2}.
\end{equation}
Therefore, inequality \eqref{eq:desired bound} follows from \eqref{eq:dimensions}, \eqref{eq:1st term} and \eqref{eq:2nd term}.
\end{proof}

\subsection{Proof of Lemma \ref{lemma:conjugation}}
\label{sec:proof conjugation}

Given $\ep>0$, let $\bPsi_\ep: [0,1/2] \to \Sp(2n+2)$ be the map given by
\[
\bPsi_\ep(t)(z_1,\dots,z_{n+1}) = (e^{2\pi\sqrt{-1}(\theta_1-\ep) t}z_1,\dots,e^{2\pi\sqrt{-1}(\theta_{n+1}-\ep) t}z_{n+1}).
\]
Then a straightforward computation shows that
\[
\cz(\bPsi_\ep) = 0
\]
and
\[
\cz(\bPsi_\ep^2) = -(n+1),
\]
whenever $\ep$ is sufficiently small. Thus, since for $\ep>0$ sufficiently small $\bPsi_\ep(1/2)$ does not have eigenvalue $-1$,
\[
\Bott^+_{\bPsi_\ep}(-1) = \Bott_{\bPsi_\ep}(-1) = -(n+1).
\]
Let $\vr: S^{2n+1} \to S^{2n+1}$ be a contactomorphism. As in the statement of the lemma, consider the corresponding lifted symplectomorphism $\psi: \R^{2n+2}\setminus\{0\}\to \R^{2n+2}\setminus\{0\}$, the conjugated action $\phi'_t:=\psi^{-1}\phi_t\psi$ and the symplectic path $\Phi'(t):=d\phi'_t(\psi^{-1}(x))$ given by the linearized circle action along the half orbit $\phi'|_{[0,1/2]}(\psi^{-1}(x))$ (recall that $x=\ga(0)$). We will show that for every $\ep>0$ sufficiently small there exists $\vr$ such that $\Phi'$ is homotopic with fixed endpoints to $\bPsi_\ep$ so that
\[
\Bott^+_{\Phi'}(-1) = -(n+1).
\]
Moreover, $\vr$ is $C^1$-arbitrarily close to the identity as $\ep\to 0$.

In order to construct $\vr$, consider the Hamiltonian $F: \R^{2n+2} \to \R$ given by
\[
F(z_1,\dots,z_{n+1}) = \pi\sum_{i=1}^{n+1} \|z_i\|^2.
\]
Let $\beta: S^{2n+1} \to [0,1]$ be a smooth bump function such that $\beta(y)=1$ for every $y$ in a neighborhood $U$ of $-x/\|x\|$ and $\beta(y)=0$ for all $y \notin V$, where $V$ is a neighborhood of $-x/\|x\|$ such that $U \subset V$ and $x/\|x\| \notin V$. Let $\lambda: \R^{2n+2}\setminus\{0\} \to \R$ be the smooth function given by $\lambda(cy)=c^2\beta(y)$ for every $y \in S^{2n+1}$ and $c \in (0,\infty)$. 

Let $G: \R^{2n+2}\setminus\{0\} \to \R$ be the Hamiltonian given by $\lambda F$. Let $\epsilon>0$ and set $\psi=\vr^G_\ep$, where $\vr^G_t$ denotes the Hamiltonian flow of $G$. By construction, $G$ equals $F$ on a neighborhood of $-x$ and vanishes on $x$. Therefore, taking $\epsilon$ sufficiently small, we have that $\psi:=\vr^F_\ep$ satisfies
\[
\psi(z_1,\dots,z_{n+1})=(e^{2\pi\sqrt{-1}\ep}z_1,\dots,e^{2\pi\sqrt{-1}\ep}z_{n+1})
\]
for every $(z_1,\dots,z_{n+1})$ close to $-x$ and $\psi$ is identity on a neighborhood of $x$. In particular, $d(\psi^{-1})(-x) = R_\ep$, where
\[
R_\ep(z_1,\dots,z_{n+1})=(e^{-2\pi\sqrt{-1}\ep}z_1,\dots,e^{-2\pi\sqrt{-1}\ep}z_{n+1})
\]
and $d\psi(x)=\Id$. Since $G$ is homogeneous of degree two, $\psi$ is the lift of a contactomorphism $\vr: S^{2n+1} \to S^{2n+1}$. Moreover, we can make $\vr$ arbitrarily $C^1$-close to identity taking $\ep \to 0$. (Actually, we can take $\vr$ arbitrarily $C^\infty$-close to identity.)

Now, consider the conjugated action $\phi'_t:=\psi^{-1}\phi_t\psi$. The symplectic path $\Phi'(t)=d\phi'_t(\psi^{-1}(x))$ starts at the identity and is $C^0$-close to the path $\Phi(t)=d\phi_t(\psi^{-1}(x))$. Since $\bPsi_\ep$ also starts at the identity and is $C^0$-close to $\Phi$, in order to show that  $\Phi'$ is homotopic with fixed endpoints to $\bPsi_\ep$ it is enough to show that $\Phi'(1/2)=\bPsi_\ep(1/2)$. But
\begin{align*}
\Phi'(1/2) & = d(\psi^{-1})(\phi_{1/2}(x))\circ d\phi_{1/2}(\psi^{-1}(x))\circ d\psi(\psi^{-1}(x)) \\
& = d(\psi^{-1})(-x)\circ (-\Id)\circ d\psi(x) \\
& = -R_\ep \\
& = \bPsi_\ep(1/2),
\end{align*}
where the second and third equality follow from the construction of $\psi$ and the fact that $\phi_{1/2}=-\Id$.

\end{document}